\documentclass[11pt]{amsart}
\usepackage[margin=1.55in]{geometry}

\pdfoutput=1

\usepackage{amsmath, amscd}
\usepackage{amsfonts}
\usepackage{amssymb}
\usepackage{amsthm}
\usepackage{upref}
\usepackage{multirow}
\usepackage{graphicx}  
\usepackage{subfig}            
\usepackage{url}               
\usepackage{mathtools}
\usepackage{pinlabel}
\usepackage{color}
\usepackage{hyperref}
\usepackage{changepage}
\usepackage{epstopdf}
\usepackage{enumerate}

\newcommand{\Z}{\mathbb{Z}}

\newcommand{\Q}{\mathbb{Q}}
\newcommand{\Out}{\operatorname{Out}}
\newcommand{\Aut}{\operatorname{Aut}}
\newcommand{\st}{\ensuremath{\operatorname{st}}}
\newcommand{\lk}{\ensuremath{\operatorname{lk}}}

\newcommand{\SL}{\operatorname{SL}}
\newcommand{\GL}{\operatorname{GL}}
\newcommand{\PGL}{\operatorname{PGL}}
\newcommand{\Sout}{\operatorname{SOut}}
\newcommand{\Saut}{\operatorname{SAut}}
\newcommand{\IA}{\operatorname{IA}}

\newcommand{\mids}{\, | \, }
\newcommand{\m}{\ensuremath{^{-1}}}
\newcommand{\ad}{\operatorname{ad}}
\renewcommand{\phi}{\varphi}
\newcommand{\cal}[1]{\mathcal{#1}}
\newcommand{\grp}[1]{\langle #1 \rangle}

\newcommand{\im}{\operatorname{im}}
\newcommand{\Id}{\operatorname{Id}}

\newcommand{\Span}{\operatorname{Span}}
\newcommand{\z}{\mathbf{z}}
\newcommand{\x}{\mathbf{x}}
\newcommand{\y}{\mathbf{y}}
\newcommand{\w}{\mathbf{w}}

\renewcommand{\v}{\mathbf{v}}

\newtheorem{thm}{Theorem}[section]
\newtheorem{thmspecial}{Theorem}
\newtheorem{corspecial}[thmspecial]{Corollary}

\newtheorem{propspecial}[thmspecial]{Proposition}

\newtheorem{lem}[thm]{Lemma}
\theoremstyle{remark}
\newtheorem{rem}[thm]{Remark}
\theoremstyle{definition}
\newtheorem{defn}[thm]{Definition}

\newtheorem{conv}[thm]{Convention}
\theoremstyle{plain}
\newtheorem{prop}[thm]{Proposition}
\theoremstyle{definition}

\newtheorem{quess}[thmspecial]{Question}
\theoremstyle{definition}

\newcounter{excounter}
\newtheorem{ex}[excounter]{Example}

\newcounter{stepcounter}
\newtheorem{step}[stepcounter]{Step}

\begin{document}

\title[On virtual indicability and property (T) for $\Out(A_\Gamma)$]{On virtual indicability and property (T) for outer automorphism groups of RAAGs}

\author{Andrew Sale}

\maketitle

    \begin{abstract}
    We give a condition on the defining graph of a right-angled Artin group which implies its automorphism group is virtually indicable, that is, it has a finite-index subgroup  that admits a homomorphism onto $\Z$.
    
    We use this as part of a criterion that determines  precisely when the outer automorphism group of a right-angled Artin group defined on a graph with no separating intersection of links has property (T).
    As a consequence we also obtain a similar criterion for graphs in which each equivalence class under the domination relation of Servatius generates an abelian group.
    \end{abstract}

\section{Introduction}

Kazhdan's property~(T) is a rigidity property for groups with many wide-ranging applications, as discussed in the introduction to \cite{PropT}.
A group is said to have property~(T) if every unitary representation with almost invariant vectors has invariant vectors.
It is notoriously challenging to prove that a group has property~(T), however there are certain obstructions to it that can more readily be demonstrated.
One such obstruction is that of \emph{virtual indicability}, which occurs for a group when it has a finite index subgroup that admits a surjection onto $\Z$.

Recall that a right-angled Artin group (RAAG) is defined by a presentation usually determined by a simplicial graph $\Gamma$.
The vertex set of $\Gamma$ provides the generating set for this presentation, while the defining relators come from commutators between all pairs of adjacent vertices in $\Gamma$.
The RAAG associated to the graph $\Gamma$ is denoted $A_\Gamma$.

In the universe of finitely presented groups, the outer automorphism groups of RAAGs provide a bridge between the groups $\GL(n,\Z)$ and $\Out(F_n)$, the outer automorphism groups of the non-abelian free groups.
At each head of the bridge we understand the groups' behaviours with regards to property (T).
For $n\geq 3$ it is well-known that $\GL(n,\Z)$ has property (T), while $\GL(2,\Z)$ does not;
at the other end, for $n\geq 5$,  computer-aided proofs of Kaluba, Kielak, Nowak and Ozawa \cite{KNO_propT,KKN_propT} tell us $\Aut(F_n)$ (and hence $\Out(F_n)$ since property~(T) is inherited by quotients) have property (T), while $\Out(F_2)$ and $\Out(F_3)$ do not \cite{GrunewaldLubotzky:AutFree}.
A recent preprint of Nitsche has extended the proof of property (T) for $\Aut(F_n)$ to $n=4$ \cite{Nitsche-propT}.

Besides property (T), there are many characteristics that are shared by both $\GL(n,\Z)$ and $\Out(F_n)$, such as finite presentability, residual finiteness, finite virtual cohomological dimension, satisfying a Tits alternative, or being of type VF.
All these mentioned hold in fact for $\Out(A_\Gamma)$ for all RAAGs $A_\Gamma$ \cite{DayPeakReduction,CV09,CV11,HorbezTA,DayWade}.
Other properties hold at both ends of the bridge, but somewhere in between they may fail. For example, $\Out(F_n)$ and $\GL(n,\Z)$ are of course infinite groups (for $n\geq 2$), and both contain non-abelian free subgroups, but there are non-cyclic RAAGs $A_\Gamma$ for which $\Out(A_\Gamma)$ is finite (for example, take the defining graph $\Gamma$ to be a pentagon), or which are infinite but do not contain free subgroups \cite{Day_solvable}.
Another example is having a finite outer automorphism group of $\Out(A_\Gamma)$. Both $\Out(\Out(F_n))$ and $\Out(\GL(n,\Z))$ are finite \cite{BV-finiteoutout,HuaReiner}, but this is not always true for $\Out(\Out(A_\Gamma))$ \cite{BregmanFullarton}.

In this paper we address the following.
\begin{quess}\label{q:T}
	For which $\Gamma$ does $\Out(A_\Gamma)$ satisfy Kazhdan's property (T)?
\end{quess}

Beyond the above results for $\GL(n,\Z)$ and $\Out(F_n)$, there are a number of partial answers to this question.
For example, Aramayona and Mart\'inez-P\'erez give conditions (see Theorem~\ref{thms:AMP} below) that imply virtual indicability, and hence deny property~(T) \cite{Aramayona-MartinezPerez}.
We can also describe precisely when $\Out(A_\Gamma)$ is finite (and so has property~(T)), and when it is infinite virtually nilpotent (and so it does not have property~(T)) \cite{Day_solvable}.
See Section~\ref{sec:what is known} for more details on what is known and what is open.

The objective of this paper is to expand one of the conditions given by \cite{Aramayona-MartinezPerez} and use it to definitively answer Question~\ref{q:T} whenever $\Gamma$ contains no \emph{separating intersection of links} (SILs).
These occur when all paths from some vertex $z$ of $\Gamma$ to either of a non-adjacent pair $x,y$ must pass through a vertex that is adjacent to both $x$ and $y$, i.{}e.{} the path hits the intersection of the links of $x$ and $y$, $\lk(x)\cap\lk(y)$ (see Definition~\ref{d:SIL}).

Having a SIL is precisely the condition required for $\Out(A_\Gamma)$ to contain two non-commuting partial conjugations (automorphisms defined by conjugating certain vertices in $\Gamma$ by a given vertex, see Section~\ref{sec:prelim:gen}). 
One way to interpret the meaning of a SIL is that its absence pushes $\Out(A_\Gamma)$ to behave more like $\GL(n,\Z)$.
If $A_\Gamma = F_n$, for $n\geq 3$, then any three vertices of $\Gamma$ form a SIL, while if $A_\Gamma = \Z^n$, then $\Gamma$ has no SIL. For the latter, all partial conjugations are clearly trivial, while for the former they play a crucial role in the structure of $\Out(F_n)$.
This interpretation is strengthened with the following two observations. 
With no SIL, a finite index subgroup of $\Out(A_\Gamma)$ admits a quotient by a finite-rank free abelian subgroup that is a block-triangular matrix group with integer entries.
This quotient is from the standard representation of $\Out(A_\Gamma)$, obtained by acting on the abelianisation of $A_\Gamma$.
The second observation is that, again with no SIL, a finite index subgroup of $\Out(A_\Gamma)$ admits a quotient by a finitely generated nilpotent group that is a direct product of groups $\SL(n_i,\Z)$, for various integers $n_i$ \cite[Theorem 2]{GuirardelSale-vastness}.

We highlight also a second feature of $\Gamma$ that effects the automorphism group of the corresponding RAAG.
The \emph{domination relation} of Servatius \cite{Servatius} determines when a transvection $R_u^v$, defined for vertices $u,v$ of $\Gamma$ by sending $u$ to $uv$ and fixing all other vertices, is an automorphism of $A_\Gamma$.
We say $u$ is dominated by $v$, and write $u \leq v$, if any vertex adjacent to $u$ is also adjacent or equal to $v$ (i.e.{} the link of $u$ is contained in the star of $v$: $\lk(u)\subseteq \st(v)$).
We have that  $u\leq v$ if and only if $R_u^v\in \Aut(A_\Gamma)$.
The relation of domination is a preorder and therefore determines equivalence classes of vertices.

The following  is the aforementioned criteria of Aramayona and Mart\'inez-P\'erez to deny property~(T).

\begin{thmspecial}[Aramayona--Mart\'inez-P\'erez \cite{Aramayona-MartinezPerez}]\label{thms:AMP}
Consider the following properties of a graph $\Gamma$:
	\begin{enumerate}
	\renewcommand{\theenumi}{\normalfont (A\arabic{enumi})}
	\renewcommand\labelenumi{\theenumi}
	\item\label{A} if $u,v$ are distinct vertices of $\Gamma$ such that $u\leq v$ then there exists a third vertex $w$ such that $u\leq w\leq v$.
	\item\label{B} if $v$ is a vertex such that $\Gamma\setminus \st(v)$ has more than one connected component, then there exists a vertex $u$ such that $u\leq v$.
\end{enumerate}	
	Let $\Gamma$ be any simplicial graph. 
	If either property \ref{A} or \ref{B} fails, then $\Aut(A_\Gamma)$ has a finite index subgroup that maps onto $\Z$.
\end{thmspecial}

(In \cite{Aramayona-MartinezPerez}, condition~\ref{A} is referred to by (B2), while condition~\ref{B} is the hypotheses of Theorem~1.6.)

When condition~\ref{A} fails, there are vertices $u,v$ in $\Gamma$ such that there is no $w$ satisfying $u\leq w\leq v$.
A homomorphism from a finite index subgroup of $\Aut(A_\Gamma)$ can be constructed with image $\Z$ so that the transvection $R_u^v$ has infinite order image.
The failure of condition~\ref{B} is used to give a surjection onto $\Z$ by exploiting a partial conjugation with multiplier $v$ whose star separates $\Gamma$ into two or more connected components, but which does not dominate any other vertex.

The first main task of this paper is to modify condition \ref{B}, defining a condition \ref{B'}
to obtain more graphs $\Gamma$ for which $\Aut(A_\Gamma)$ is virtually indicable.
The key, as with \ref{B}, is to exploit certain partial conjugations.

We say a graph $\Gamma$ satisfies condition~\ref{B'} if for every vertex $x$ and component $C$ of $\Gamma\setminus \st(x)$ some non-zero power of the partial conjugation by $x$ on $C$ can be expressed as a product in $\Out(A_\Gamma)$ of partial conjugations (or their inverses) by $x$ with supports that are components of $\Gamma \setminus \st(y)$ for some $y$ dominated by, but not equal to, $x$.
See the start of Section~\ref{sec:virt ind} for another definition, and Lemma~\ref{lem:B' simply} for equivalence with the one given here.
One can observe that if \ref{B'} holds, then necessarily \ref{B} holds also.

Our first main result is the following.

\begin{thmspecial}\label{thm:no T special}
	If a simplicial graph fails property~\ref{B'} then $\Aut(A_\Gamma)$ has a finite index subgroup that admits a surjection onto $\Z$.
\end{thmspecial}

We use this to answer Question~\ref{q:T} entirely when the defining graph has no SIL.

\begin{thmspecial}\label{thms:T}
	Suppose $\Gamma$ has no SIL.
	Then $\Out(A_\Gamma)$ has property (T) if and only if both properties \ref{A} and \ref{B'} hold in $\Gamma$.
\end{thmspecial}

This also allows us to give an answer in any case when all equivalence classes in $\Gamma$ are abelian.

\begin{corspecial}\label{cors:T}
	Suppose all equivalence classes in $\Gamma$ are abelian.
	Then $\Out(A_\Gamma)$ has property~(T) if and only if
	$\Gamma$ has no SIL and  both properties \ref{A} and \ref{B'} hold.
\end{corspecial}

\begin{proof}
	If $\Gamma$ does contain a SIL then $\Out(A_\Gamma)$ is large by \cite[Theorem~2]{GuirardelSale-vastness}, and hence does not have property~(T).
	If $\Gamma$ does not contain a SIL then we apply Theorem~\ref{thms:T}.
\end{proof}

The proof of Theorem~\ref{thm:no T special} involves a composition of restriction and projection maps (see Section~\ref{sec:prelim:restriction projection}) to focus our attention on a smaller portion of $\Gamma$, 
ultimately mapping into the automorphism group of a free product of free abelian groups.
We then apply a homological representation by acting on a certain cover of the Salvetti complex of the free product.
The image of the composition of all these maps can be seen to admit a surjection onto $\Z$.

Once Theorem~\ref{thm:no T special} is established, to prove Theorem~\ref{thms:T} we use the standard representation of $\Out(A_\Gamma)$, obtained by acting on the abelianisation of $A_\Gamma$.
This gives a short exact sequence such that, when $\Gamma$ contains no SIL, the kernel $\IA_\Gamma$, sometimes called the Torelli subgroup, is free abelian.
To fully exploit this structure of $\Out(A_\Gamma)$, we need this sequence to be split,
which we show is (virtually) the case.

In the following, we denote by $\Sout^0(A_\Gamma)$ the subgroup of $\Out(A_\Gamma)$ of finite index that is generated by the set of all transvections and partial conjugations of $A_\Gamma$. 

\begin{propspecial}\label{propspecial}
	The standard representation of $\Sout^0(A_\Gamma)$ gives the short exact sequence
	$$1\to \IA_\Gamma \to \Sout^0(A_\Gamma) \to Q \to 1$$
	which splits if 
	$\Gamma$ has no SIL.
\end{propspecial}

We note that the short exact sequence may be split even if there is a SIL. This is discussed in more detail in Remark~\ref{rem:split no sil}, where weaker sufficient conditions are given for this to occur.

Compare Proposition~\ref{propspecial} with \cite[Theorem~1.2]{GPR_automorphisms} and \cite{Tits:Coxter} where automorphism groups of certain graph products are expressed as a semidirect product in a manner similar to Propostion~\ref{propspecial}.

Finally, we comment on the fact that we show in Theorem~\ref{thms:T} that $\Out(A_\Gamma)$ has property (T), but prove nothing about $\Aut(A_\Gamma)$.
We use the fact that when there is no SIL partial conjugations commute in $\Out(A_\Gamma)$.
When dealing with automorphisms (not outer) this fact fails.
We use a decomposition of $\IA_\Gamma$ in Section~\ref{sec:T} to prove Theorem~\ref{thms:T} which involves subgroups, denoted $A_C^X$, that are normal in $\Sout^0(A_\Gamma)$.
Since they are subgroups of $\IA_\Gamma$,  they are free abelian.
This enables us to build $\Sout^0(A_\Gamma)$ out of block-triangular groups, for each of which we can verify property~(T) via a criterion of Aramayona and Mart\'inez-P\'erez \cite{Aramayona-MartinezPerez}.
When dealing with this situation for $\Saut^0(A_\Gamma)$, the groups $A_C^X$ need to include some inner automorphisms, and they no longer remain free abelian in general.

\subsection{The (un)resolved cases}
\label{sec:what is known}

We finish the introduction by summarising what is known, and what is left unknown with regards to property~(T) for outer automorphism groups of RAAGs.

\subsubsection*{Properties of $\Gamma$ that deny property~(T) in $\Out(A_\Gamma)$:}

\begin{itemize}
	\item Condition~\ref{A} fails (so $\Out(A_\Gamma)$ is virtually indicable) \cite[Corollary~1.4]{Aramayona-MartinezPerez}.
	\item Condition~\ref{B'} fails (so $\Out(A_\Gamma)$ is virtually indicable) (Theorem~\ref{thm:no T special}).
	\item If there is a non-abelian equivalence class of size three (so $\Out(A_\Gamma)$ is large) \cite[Theorem 6]{GuirardelSale-vastness}.
	\item If there is a ``special SIL'' (so $\Out(A_\Gamma)$ is large) \cite[Proposition~3.15]{GuirardelSale-vastness}. A special SIL is a SIL  $(x_1,x_2\mids x_3)$ such that
	\begin{itemize}
		\item each $x_i$ is in an abelian equivalence class,
		\item if $x_i \leq u \leq x_j$ then $u \in [x_1]\cup[x_2]\cup[x_3]$,
		\item if $u\leq x_i$, for any $i$, then there is a connected component $C$ of $\Gamma\setminus \st(u)$ so that $x_1,x_2,x_3 \in C\cup \st(u)$.
	\end{itemize}
\end{itemize}

There are a handful of cases that are covered by other means, but are contained within one of the above situations. We explain them now.

Firstly, if there is an equivalence class of size two then $\Out(A_\Gamma)$ is large and we do not have property~(T) \cite[Theorem 6]{GuirardelSale-vastness}. 
In this case we'd also have the failure of condition~\ref{A}.

We also know that $\Out(A_\Gamma)$ is virtually nilpotent if and only if there is no SIL, and
	 all equivalence classes are of size one \cite[Theorem~1.3]{Day_solvable}.
If furthermore there is either at least one vertex $v$ whose star separates $\Gamma$, or at least one pair of vertices $u,v$ satisfying $u\leq v$, then $\Out(A_\Gamma)$ must be infinite.
Finitely generated virtually nilpotent groups have property~(T) if and only if they are finite, so this gives a class of graphs where $\Out(A_\Gamma)$ does not have property~(T).
However, if all equivalence classes are of size one and there is some pair of vertices $u,v$ with $u\leq v$, then condition~\ref{A} necessarily fails.
Meanwhile, if there is no such pair $u,v$, but there is some vertex $x$ whose star separates $\Gamma$, then $x$ is minimal and so condition~\ref{B}  fails (and hence also \ref{B'}).

A third situation worth remarking upon is that if all equivalence classes are abelian and $\Gamma$ contains a SIL, then it contains a special SIL \cite[Proposition~3.9]{GuirardelSale-vastness}.
This is used for Corollary~\ref{cors:T}.

\subsubsection*{Properties of $\Gamma$ that imply $\Out(A_\Gamma)$ has property~(T):}

\begin{itemize}
	\item If $\Gamma$ has no SIL and satisfies conditions~\ref{A} and~\ref{B'} (Theorem~\ref{thms:T}).
	\item If $\Gamma$ has no edges and at least four vertices (i.e. $\Out(A_\Gamma)  = \Out(F_n)$ for $n\geq 4$) \cite{KNO_propT,KKN_propT,Nitsche-propT}.
\end{itemize}

The first point here covers both the case when $\Out(A_\Gamma) = \GL(n,\Z)$, for $n\geq 3$, as well as when $\Out(A_\Gamma)$ is finite.
The latter occurs if and only if 
 there are no vertices $x,y$ in $\Gamma$ so that $x\leq y$, and
 there is no vertex whose star separates $\Gamma$ into two or more components
 (see for example \cite[\S 6]{CF12}).

It is tempting to speculate that the no SIL condition could be removed from Theorem~\ref{thms:T}, to answer Question~\ref{q:T} completely.
However, if we allow $\Gamma$ to contain a SIL, and assume both \ref{A} and \ref{B'} hold,
then there is nothing to stop us have a free equivalence class of size three, or all equivalence classes be abelian.
In both cases, we do not have property~(T) by \cite[Theorem~6]{GuirardelSale-vastness}.

\subsection*{Paper structure}

We begin with some preliminary material in Section~\ref{sec:prelims}.
In Section~\ref{sec:VI homo rep} we establish a homological representation for a free product of free abelian groups. This representation is used in obtaining surjections onto $\Z$ later.
A new notion is introduced in Section~\ref{sec:principal} that describes which partial conjugations are necessary to virtually generate the Torelli subgroup $\IA_\Gamma$ when there is no SIL.
Theorem~\ref{thm:no T special} is proved in Section~\ref{sec:virt ind}, which is by some way the most involved part of the paper.
In Section~\ref{sec:standard rep} we establish that when there is no SIL we can express $\Sout^0(A_\Gamma)$ as a semidirect product $Q\ltimes \IA_\Gamma$, proving Proposition~\ref{propspecial}.
The proof of Theorem~\ref{thms:T} is given in Section~\ref{sec:T}.

\subsection*{Acknowledgements}
The author 
would like to thank Ric Wade for very helpful comments on a draft of this paper,
and Matthew Day, Erik Guentner and Rufus Willett for helpful comments and discussions.
He also acknowledges the support of the Simons Foundation (Simons Collaboration Grant 711986).

\section{Generators and SILs}\label{sec:prelims}

We begin with some background material on RAAGs and their automorphism groups.

RAAGs are defined via a simplicial graph $\Gamma$. A generating set of the group $A_\Gamma$ is the vertex set $V(\Gamma)$, and presentation is given by
$$A_\Gamma = \langle V(\Gamma) \mid \textrm{$[u,v]=1$ if $u$ and $v$ are adjacent in $\Gamma$}\rangle .$$

Given a simplicial graph $\Gamma$ and a vertex $v$ of $\Gamma$, the \emph{link} of $v$ is the induced subgraph on the set of all vertices connected to $v$ by an edge. It is denoted $\lk(v)$. The \emph{star} of $v$, written $\st(v)$ is the join of $\lk(v)$ with the vertex $v$.

\subsection{Domination and equivalence classes of vertices}\label{sec:prelim:dom}

This section deals with transvections on a RAAG.
We denote a transvection by $R_u^v$, for vertices $u,v\in \Gamma$, where $R_u^v(u) = uv$ and all other vertices are fixed.
Servatius made the following definition that determines when transvections are automorphisms \cite{Servatius}.

\begin{defn}\label{d:domination}
	The Servatius \emph{domination} relation $\leq$ on vertices of $\Gamma$ is given by $u\leq v$ if and only if $\lk(u)\subseteq \st(v)$.
\end{defn}

We leave as an exercise to the reader to prove that a transvection $R_u^v$ is an automorphism if and only if $u\leq v$.

\begin{rem}\label{rem:domination}
This definition can be tightened up as follows:
 if $u,v$ are adjacent then $u\leq v$ if and only if $\st(u)\subseteq \st(v)$;
 if $u,v$ are not adjacent then $u\leq v$ if and only if $\lk(u)\subseteq \lk(v)$.
 \end{rem}

The domination relation is a preorder and determines \emph{equivalence classes} of vertices.
We denote the equivalence class containing a vertex $x$ by $[x]$.
These classes come in two flavours.
They are either \emph{abelian} if all vertices in the class are pairwise adjacent, or \emph{non-abelian} (also referred to as \emph{free}) otherwise.
It is a straight-forward exercise to see that the vertices of a non-abelian equivalence class share no edges and so generate a non-abelian free group (see \cite[Lemma~2.3]{CV09}).

For subsets $X,Y$ in $\Gamma$ we write $X\leq Y$ whenever $x\leq y$ for each $x\in X$ and $y\in Y$. 
When $X$ and $Y$ are equivalence classes, if $x\leq y$ for any $x\in X$ and $y\in Y$, then $X\leq Y$.

\subsection{Generating $\Aut(A_\Gamma)$}\label{sec:prelim:gen}

The Laurence--Servatius generating set of $\Aut(A_\Gamma)$ is a finite generating set that consists of the automorphisms of $A_\Gamma$ of the following four types \cite{Laurence-thesis}.

\begin{itemize}
	\item \emph{Involutions}: automorphisms that send one generator $v$ to its inverse $v\m$ and fix all others;
	\item \emph{Graph symmetries}: automorphisms that permute the set of vertices of $\Gamma$ according to a symmetry of the graph;
	\item \emph{Transvections}: $R_u^v$ for distinct vertices satisfying $u\leq v$;
	 \item \emph{Partial conjugations}: for a vertex $v$ and a connected component $C$ of $\Gamma\setminus \st(v)$, the partial conjugation $\pi^v_C$ sends every vertex $u$ of $C$ to $v\m uv$ and fixes all others.
\end{itemize}

For transvections $R_u^v$ and partial conjugations $\pi^v_C$ we call the vertex $v$ the \emph{multiplier} and $u$ or $C$ respectively the \emph{support}.
When dealing with partial conjugations, we will also let the support $C$ denote a union of connected components of $\Gamma \setminus \st(v)$.

We denote by $\operatorname{SAut}^0(A_\Gamma)$ (respectively $\Sout^0(A_\Gamma)$) the finite-index subgroup of $\Aut(A_\Gamma)$ (respectively $\Out(A_\Gamma)$) that is generated by partial conjugations and transvections.

\subsection{Restriction and projection maps}
\label{sec:prelim:restriction projection}

In constructing virtual surjections to $\Z$, we use a composition of homomorphisms, most of which are of one of two types, restriction maps and projection (or factor) maps.
These have been exploited in the study of automorphisms  of RAAGs and other graph products before, for example \cite{CV09,CV11,GuirardelSale-vastness,SaleSusse,DayWade}.
We give a brief description of them here.

Let $\Lambda$  be a subgraph of $\Gamma$.
We can define a \emph{restriction map} to $\Out(A_\Lambda)$ from the relative outer automorphism group $\Out(A_\Gamma ; A_\Lambda)$---that is the subgroup of $\Out(A_\Gamma)$ consisting of those automorphisms that preserve the subgroup $A_\Lambda$ up to conjugacy:
$$
r \colon \Out(A_\Gamma ; A_\Lambda) \to \Out(A_\Lambda).$$
Each outer automorphism $\Phi$ in $\Out(A_\Gamma ; A_\Lambda)$ restricts to an outer automorphism $r(\Phi)$ of $A_\Lambda$.

We wish to use this when the relative outer automorphism group contains certain subgroups.
Let $X$ be a set consisting of partial conjugations and transvections.
If $X$ contains all of these then $\langle X\rangle =\Sout^0(A_\Gamma)$, but in general  this may not be the case.
The following two tests can be performed to check whether $\langle X\rangle $ is contained in $\Out(A_\Gamma ; A_\Lambda)$, and hence whether $r$ may be defined on $\langle X\rangle $:
\begin{itemize}
	\item If $R_u^v \in X$, and $u\in \Lambda$, then $v\in \Lambda$.
	\item If $\pi^w_C \in X$ and $w \notin \Lambda$, then either $\Lambda \subseteq C \cup \st(w)$ or $\Lambda \cap C = \emptyset$.
\end{itemize}

Next we discuss \emph{projection maps} (also called factor maps).
These are homomorphisms 
$$
p \colon \langle X\rangle \to \Out(A_\Lambda)$$
obtained by killing vertices of $\Gamma$ that are not in $\Lambda$.
To be more explicit, let $\kappa \colon A_\Gamma \to A_\Lambda$ be the quotient map obtained by deleting the vertices in $\Gamma\setminus \Lambda$.
Let $\phi$ be an automorphism in the outer automorphism class $\Phi$.
Then $\Phi$ is sent the the outer automorphism class of the map $\kappa(g) \mapsto  \kappa \circ\phi (g)$ for $g\in A_\Gamma$.
Thus, for $p$ to be well-defined we need the kernel of $\kappa$ to be preserved, up to conjugacy by $\langle X\rangle $.
This is true for each partial conjugation, so one just needs to check it for transvections in $X$:
\begin{itemize}
	\item If $R_u^v \in X$ and $v\in \Lambda$, then $u \in \Lambda$.
\end{itemize}

\subsection{Separating intersection of links}

Separating intersections of links (SILs) are an important feature of $\Gamma$ in relation to the properties of $\Out(A_\Gamma)$.
In particular, with no SIL, the subgroup of $\Out(A_\Gamma)$ generated by partial conjugations is abelian.
Whereas, with a SIL, this subgroup contains a non-abelian free subgroup.

\begin{defn}\label{d:SIL}
	A \emph{separating intersection of links} (SIL) is a triple of vertices $(x,y\mids z)$ in $\Gamma$ that are not pairwise adjacent, and such that the connected component of $\Gamma \setminus (\lk(x) \cap \lk(y))$ containing $z$ does not contain either $x$ or $y$.
\end{defn}

The key consequence of $\Gamma$ admitting a SIL $(x,y\mids z)$ is that, if $Z$ is the connected component of $\Gamma \setminus (\lk(x) \cap \lk(y))$ containing $z$, then $\pi^x_Z$ and $\pi^y_Z$ generate a non-abelian free subgroup of $\Out(A_\Gamma)$.

There is a relationship between SILs and equivalence classes. Notably, any three vertices in a non-abelian equivalence class determine a SIL.
Thus when considering graphs without a SIL we immediately remove the possibility of admitting non-abelian equivalence classes of size at least 3.
Furthermore, if a graph $\Gamma$ satisfies condition~\ref{A} of Theorem~\ref{thms:T} and has no SIL, then it cannot have any equivalence class of size 2 (this is immediate from condition~\ref{A}), and so all its equivalence classes must be abelian.

We end this section with some preliminary lemmas concerning SILs, particularly how they behave with respect to the domination relation.

\begin{lem}\label{lem:domination pairs}
	Suppose $u,v,w$ are distinct vertices of $\Gamma$ such that $u\leq v,w$ and $(v,w\mids u)$ is not a SIL.
	
	Then $[v,w] = 1$. 
\end{lem}

\begin{proof}
	Suppose $[v,w]\ne 1$. If $u$ and $v$ are adjacent, then $v\in \lk(u) \subseteq \st(w)$, contradicting $[v,w]\ne 1$.
	Thus $[u,v]\ne 1$, and similarly $[u,w]\ne 1$.
	Then we have $\lk(u) \subseteq \lk(v)\cap\lk(w)$ by Remark~\ref{rem:domination}, implying that $(v,w\mids u)$ is a SIL.
\end{proof}

We want to understand how partial conjugations behave under conjugation by transvections.
When there is no SIL, either the partial conjugation and transvection commute, or the conjugation action itself behaves like a transvection---see Lemma~\ref{lem:conjugates of pc by transv}.
Before stating this, we quickly note the following.

\begin{lem}\label{lem:dominating multiplier in pc}
	Suppose $x\leq y$ and $C$ is a connected component of $\Gamma \setminus \st(x)$. 
	
	Then $C ' = C \setminus \st(y)$ is a (possibly empty) union of connected components of $\Gamma\setminus \st(y)$.
\end{lem}

\begin{proof}
	Suppose $z$ is any vertex of $\Gamma$ not in $C$ or $\st(y)$.
	If there is a path between $z$ and a vertex of $C$, then it must pass through $\st(x)$, and hence through $\lk(x)$.
	Since $x\leq y$, the path therefore intersects $\st(y)$.
\end{proof}

Note that if $C$ is empty, then by convention we understand $\pi^x_C$ to mean the identity map.

\begin{lem}\label{lem:conjugates of pc by transv}
	Let $v,x,y\in \Gamma$ be such that $x\leq y$, and let $C$ be a connected component of $\Gamma \setminus \st(v)$. 
	Then, for $\varepsilon,\delta \in \{1,-1\}$, in $\Out(A_\Gamma)$,
	\begin{enumerate}
		\renewcommand{\theenumi}{(\roman{enumi})}
			\renewcommand\labelenumi{\theenumi}
		\item \label{relation1} $\pi^v_C$ and $R_x^{y}$ commute if $v \ne x$ and either $(v,y\mids x)$ is not a SIL or $x,y\notin C$ or $x,y \in C \cup \st(v)$,
		\item \label{relation3} $(\pi^v_C)^{\varepsilon} R^y_x (\pi^v_C)^{-\varepsilon} = (R^v_x)^{-\varepsilon} R^y_x (R^v_x)^\varepsilon$, if $v\ne x$, $(v,y\mids x)$ is a SIL, and exactly one of  $x$ or $y$ is in $C$,
		\item \label{relation2} $(R_x^y)^{\delta} \pi^x_C (R_x^y)^{-\delta} = \pi^{x}_C  (\pi^{y}_{C'})^{\delta}$ if $v=x$, where $C ' = C \setminus \st(y)$.
	\end{enumerate}
\end{lem}

\begin{proof}
	For part~\ref{relation1}, the conclusion when $x,y\notin C$ is immediate as the supports and multipliers are disjoint. 
	Meanwhile, if $x,y\in C\cup \st(v)$ we can multiply by an inner to get $x,y \notin C$.
	
	Now suppose that $v\ne x$ and $(v,y\mids x)$ is not a SIL.
	Up to multiplication by an inner, we may assume that $x\notin C$.
	If $v=y$ then~\ref{relation1}	holds since the supports of the transvection and partial conjugation are disjoint and the multipliers of each are fixed by one-another. So assume $v\ne y$.
	We claim that $y \notin C$ also, which implies the automorphisms commute as before.
	To prove the claim, if $x\in \st(v)$, then $x\leq y$ implies $y\in \st(v)$ too. 
	Assume $x\notin \st(v)$ and $y\in C$. Since $x\notin C$, any path from $x$ to $y$ passes through $\st(v)$, and in particular, using the fact that $x\leq y$, we must therefore have $\lk(x)\subseteq \lk(v)$. Hence $x\leq v,y$. Since $(v,y\mids x)$ is not a SIL, Lemma~\ref{lem:domination pairs} implies that $[v,y] = 1$, contradicting $y\in C$.
	
	Now assume the hypotheses for \ref{relation3} hold.
	As ($v,y\mids x)$ is a SIL we must have $x\leq v$, as $\lk(x)$ must be contained in $\lk(v)\cap\lk(y)$.
	As $C$ contains exactly one of $x$ or $y$, up to multiplication by an inner we may assume $y\in C$ and $x\notin C$.
	Direct calculation then yields
	$$	\pi^v_C R^y_x (\pi^v_C)^{-1} ( x ) = xv^{-1}yv = (R^v_x)^{-1} R^y_x R^v_x (x)$$
	while all other vertices, including those in $C$, are fixed.
	This confirms the claimed identity in \ref{relation3} when $\varepsilon=1$.
	The case when $\varepsilon=-1$ is similar.

	Now assume that $v=x$. 
	Again, up to an inner automorphism, we may assume that $y \notin C$.
	Then direct calculation (left to the reader) verifies the relation~\eqref{relation2}, with Lemma~\ref{lem:dominating multiplier in pc} ensuring $\pi^y_{C'}$ makes sense.
\end{proof}

\subsection{The standard representation}\label{sec:prelim:standard rep}

The standard representation of $\Out(A_\Gamma)$ is obtained by acting on the abelianisation of $A_\Gamma$.
We denote it by
$$\rho \colon \Out(A_\Gamma) \to \GL(n,\Z)$$
where $n$ is equal to the number of vertices in $\Gamma$.

The image of $\rho$ is described in more detail in Section~\ref{sec:standard rep}. Here we focus on the kernel.
It is denoted by $\IA_\Gamma$, and is sometimes referred to as the Torelli subgroup of $\Out(A_\Gamma)$.
Day and Wade independently proved that $\IA_\Gamma$ is generated by the set of partial conjugations and commutator transvections $R_u^{[v,w]} = [R_u^v,R_u^w]$ (see \cite[\S 3]{Day_symplectic} and \cite[\S 4.1]{WadeThesis}).
It follows that $\IA_\Gamma$ is a subgroup of $\Sout^0(A_\Gamma)$, so the kernel is unchanged when taking the restriction of $\rho$ to $\Sout^0(A_\Gamma)$.
We will abuse notation by also calling this restriction $\rho$.

The no SIL condition implies that we have no commutator transvections---this is a consequence of Lemma~\ref{lem:domination pairs}.
It is not hard to see that with no SILs all partial conjugations commute up to an inner automorphism  (it is proved in, or follows immediately from results in, each of \cite{CCV_automorphisms,GPR_automorphisms,Day_solvable,GuirardelSale-vastness}). 
In particular:

\begin{prop}\label{prop:Torelli}
	Suppose $\Gamma$ has no SIL.	
	Then $\IA_\Gamma$ is free abelian and is generated by the set of all partial conjugations.
\end{prop}

\section{A homological representation for a free product of abelian groups}
\label{sec:VI homo rep}

In this section we describe a homological representation for a finite index subgroup of $\Aut(G)$
where 
$$
G = \Z^{c_0} \ast \cdots \Z^{c_s} \ast \Z^d.
$$
This representation will be used in the proof of Theorem~\ref{thm:no T special} to obtain a virtual surjection onto $\Z$ when condition~\ref{B'} fails.
It is obtained by acting on a subspace of the homology of a certain cover of the Salvetti complex associated to $G$.

For each $\Z^{c_i}$ factor, let $Z_i$ be a basis set.
For the $\Z^d$ factor write a basis as $\cal{Y} = \{y_1,\ldots,y_k,x\}$, so that $d=k+1$.
The reason for distinguishing the element $x$ from the $y_i$'s will become clear in Section~\ref{sec:virt ind} when the homological representation is applied.

We now describe the cover of the Salvetti complex on which we act to get the representation.
Take the elements of $\Z_2$ to be $1$ (the identity) and $g$ (the non-identity element).
Let $\pi \colon G \to \Z_2$ be defined by
$$
\pi(v) = \begin{cases} 
1 & \textrm{if $v \in \cal{Y}$,}\\
g & \textrm{otherwise.}
\end{cases}$$
The Salvetti complex $S$ of $G$ is a wedge of tori, of dimensions $c_0,c_1,\ldots , c_s,d$.
Let $T$ be the double cover of $S$ on which $\Z_2$ acts by deck transformations.
The \mbox{1--skeleton} of $T$ can be identified with the Cayley graph of $\Z_2$ with generating set given by the image of $Z_0\cup\cdots Z_s \cup \cal{Y}$.
Specifically, there will be edges connecting the vertex labelled $1$ to the vertex $g$ for each $a \in Z_i$. 
We denote these edges $e_a$, and for each there is a corresponding edge $g e_a$ from $g$ to $1$.
When $a$ is in $\cal{Y}$ we have two single-edge loops, $e_a$ at the vertex $1$, and its image $g e_a$ at $g$.
The edges $e_a$ and $g e_a$ for $a\in Z_i$ form the $1$--skeleton for a $c_i$--dimensional torus, a two-sheeted cover of the corresponding torus in $S$.
The $e_a$ edges for $a\in \cal{Y}$ are the $1$--skeleton for a $d$--dimensional torus, while the edges $g e_a$ form the $1$--skeleton for a copy of this torus under $g$.

Let $\Aut_\pi(G)$ be the finite-index subgroup of $\Aut(G)$ of automorphisms $\varphi$ such that $\pi\circ\varphi = \pi$.
These automorphisms are induced by homotopy equivalences of $S$ which lift to homotopy  equivalences of $T$, fixing the two vertices $1$ and $g$.
We therefore have an action of $\Aut_\pi(G)$ on the homology $H_1(T;\Q)$, which preserves $H_1(T;\Z)$ and commutes with the action of $\Z_2$.

The action we desire is on a subspace of $H_1(T;\Q)$, namely the eigenspace corresponding to the eigenvalue $-1$ for the action of $\Z_2$.
We will denote this eigenspace by $V_{-1}$.
In the special case when $c_0=c_1=\cdots=c_s=d=1$, the group $G$ is the free group $F_{s+2}$, and by G\"aschutz \cite{Gaschutz} the homology $H_1(T;\Q)$ decomposes as $\Q \oplus \Q[\Z_2]^{s+1}$, and $V_{-1} \cong \Q^{s+1}$. In general we show:

\begin{lem}\label{lem:Gaschutz here}
	Let $T$ be the two-sheeted cover of the Salvetti complex of $G$ associated to the map $\pi\colon G \to \Z_2$ defined above, and $V_{-1}$ be the $-1$--eigenspace for the action by $\Z_2$.
	Then
	\begin{enumerate}[(I)]
		\item $H_1(T;\Q) \cong \Q^{(\sum c_i) -s} \oplus \Q[\Z_2]^{d+s}$,
		\item $V_{-1} \cong \Q^{d+s}$.
	\end{enumerate}	
\end{lem}

\begin{proof}
	Let $A_i$ denote the $\Q$--vector space of $i$--dimensional chains in $T$.
	Each of these decomposes into the sum of $+1$ and $-1$--eigenspaces under the $\Z_2$ action: $A_i = A_i^{(+1)} \oplus A_i^{(-1)}$.
	The boundary maps $\partial_i \colon A_i \to A_{i-1}$ commute with the action of $\Z_2$ so we get restrictions of these to the eigenspaces:
	$$\partial_i^{(+1)} \colon A_i^{(+1)} \to A_i^{(+1)} \ \ \textrm{and} \ \ \partial_i^{(-1)} \colon A_i^{(-1)} \to A_i^{(-1)}.$$
	The eigenspace $V_{-1}$ of $H_1(T;\Q)$ is the quotient $\ker \partial_1^{(-1)} / \im \partial_2^{(-1)}$.
	
	Fix a vertex $z_0 \in Z_0$.
	The space of all $1$-cycles has dimension $2d+2(\sum c_i)-1$ and basis given by the following vectors:
	\begin{itemize}
		\item $e_x , e_{y_1}, \ldots , e_{y_k}$ and $ge_x , ge_{y_1}, \ldots , ge_{y_k}$,
		\item $e_{z_0} - e_a$ and $g(e_{z_0}-e_a)$ for $a\in Z_0 \cup \cdots \cup Z_s$, $a\ne z_0$,
		\item $e_{z_0} + ge_{z_0}$.
	\end{itemize}
	We can describe its structure as 
	$$\ker \partial_1 \cong \Q \oplus \Q[\Z_2]^{d + (\sum c_i) - 1}.$$
	Moving onto the \mbox{$-1$--eigenspace},
	the dimension of $\ker \partial_1^{(-1)}$ is $k+(\sum c_i)$ and a basis is given by:
	\begin{itemize}
		\item $(1-g)e_x, (1-g)e_{y_1}, \ldots , (1-g) e_{y_k}$,
		\item $(1-g)(e_{z_0} - e_a)$ for $a\in Z_0 \cup \cdots \cup Z_s$, $a\ne z_0$.
	\end{itemize}

	There are two types of $2$--cells found in $T$.
	The first are those coming from the commutator relation between two vertices in $Z_i$, and the second are those from the commutator relation between two vertices in $\cal{Y}$.
	The latter have boundary equal to zero in $A_1$.
	The former have boundary $\pm (e_a + g e_b -g e_a -e_b)$, for $a,b\in Z_i$.
	In particular, $\im \partial_2$ is generated by $(1-g)(e_a-e_b)$ for $a,b\in Z_i$, and $i=0,\ldots,s$.
	We can fix a vertex $z_i$ in each of the remaining classes $Z_i$.
	Then a basis for $\im \partial_2$ is given by 
	$$\{ (1-g)(e_{z_i}-e_a) \mid a\in Z_i\setminus\{z_i\}, \  i=0,\ldots,s \}.$$
	We therefore have $\dim (\im \partial_2) = \sum (c_i - 1)$.
	Furthermore, these are all cycles in the $-1$--eigenspace of $A_1$, so $\im \partial_2 = \im \partial_2^{(-1)}$.
	We conclude that $V_{-1}$ has dimension
	$$k+\sum\limits_{i =0}^s c_i - \sum\limits_{i =0}^s (c_i - 1) = k + s + 1 = d+s$$
	as required.
	
	Finally, the claimed structure of $H_1(T;\Q)$ follows by the fact that the following list of elements form a basis:
	\begin{itemize}
		\item $e_x , e_{y_1}, \ldots , e_{y_k}$ and $ge_x , ge_{y_1}, \ldots , ge_{y_k}$,
		\item $e_{z_0} - e_{z_i}$ and $g(e_{z_0}-e_{z_i})$ for $i = 1,\ldots , s$,
		\item $e_{z_i} - e_a = g(e_{z_i} - e_a)$ for $a\in Z_i$ and $i = 0,\ldots , s$,
		\item $e_{z_0} + ge_{z_0}$.
	\end{itemize}
\end{proof}

We note that, following the proof of Lemma~\ref{lem:Gaschutz here}, we can write down a basis for $V_{-1}$ as follows:
\begin{itemize}
	\item $\x = (1-g)e_x$,
	\item $\y_j = (1-g)e_{y_j}$ for $j=1,\ldots , k$,
	\item $\z_i = (1-g)(e_{z_0} - e_{z_i})$ for $i=1,\ldots , s$.
\end{itemize}

We thus have a representation of $\Aut_\pi(G)$ obtained by acting on $V_{-1}$:
$$
\Aut_\pi(G) \to \PGL(V_{-1}).$$
We observe that inner automorphisms act on $V_{-1}$ as $-1$, so are in the kernel of this representation.
This means that it factors through the finite index subgroup $\Out_\pi(G)$ of $\Out(G)$ that is the quotient of $\Aut_\pi(G)$ by the inner automorphisms.
Thus we define the representation $\rho_\pi$ to be the representation on $\Out_\pi(G)$:
$$
\rho_\pi \colon \Out_\pi(G) \to \PGL(V_{-1}).
$$

\subsection{The action of partial conjugations}\label{sec:action of partial conj}

We now take the time to look at how partial conjugations behave under the representation $\rho_\pi$.
Recall that we are acting on the projective space associated to $V_{-1}$.

We first look at how the partial conjugations $\pi^x_{Z_i}$ act on the vectors $\z_j$.
First assume $i=0$.
Then $\pi^x_{Z_0}$ sends $z_0$ to $x^{-1} z_0 x$ in $G$.
This means that $e_{z_0}$ is sent to $-e_x + e_{z_0} + g e_x = e_{z_0} - (1-g) e_x$,
and we get, for each $j=1,\ldots, r$,
$${\rho_\pi} ( \pi^x_{Z_0} ) (\z_j) = (1-g)(e_{z_0} - (1-g)e_x - e_{z_j}) = \z_j - (1-g)^2 e_x = \z_j - 2 \x.$$
All other basis vectors are fixed by $\pi^x_{Z_0}$.
Thinking of its matrix representation, if we order the basis elements as $\z_1,\ldots,\z_r , \y_1,\ldots\y_k,\x$, then the matrix for $\pi^x_{Z_0}$ is as follows, with the first block marking off the basis vectors $\z_i$, the second block for $\y_i$ and the final block for $\x$.
\newcommand{\rvline}{\hspace*{-\arraycolsep}\vline\hspace*{-\arraycolsep}}
$$
{\rho_\pi} ( \pi^x_{Z_0} ) =
\left(
\begin{array}{cccc|cccc|c}
1 & 0 &\cdots & 0 & 0 & 0 & \cdots & 0 &  0 \\
0 & 1 & \cdots & 0 &  0 & 0 &\cdots & 0 &  0 \\
\vdots & & \vdots &  &  &\vdots & & \vdots & \\
0 & 0 & \cdots & 1 & 0 & 0 &\cdots & 0 &  0 \\
\hline 
0 & 0 & \cdots & 0 & 1 & 0 &\cdots & 0 & 0 \\
0 & 0 & \cdots & 0 & 0 & 1 &\cdots & 0 & 0 \\
\vdots & & \vdots & &  &\vdots & & \vdots & \\
0 & 0 & \cdots & 0 & 0 & 0 &\cdots & 1 & 0 \\
\hline
-2 & -2 & \cdots & -2 & 0 & 0 &\cdots & 0 &  1 \\
\end{array}\right).$$
Similar calculations yield that ${\rho_\pi}(\pi^x_{Z_i})$, for $i=1,\ldots,s$, will be a transvection:
$$\bar{\rho_\pi}(\pi^x_{Z_i}) ( \z_j ) = 
\begin{cases} 
\z_j + 2\x & \textrm{if $i=j$,}\\
\z_j & \textrm{otherwise.}
\end{cases}$$
The matrix representation for this will be an elementary matrix differing from the identity by a $2$ in the appropriate entry.

Extending this to other multipliers, 
we will use the standard inner product on $\Q^s$ to describe the action of the partial conjugation on vectors in $W = \Span_\Q(\z_1,\ldots, \z_s) \cong \Q^s$.
Consider a partial conjugation $\pi^a_D$, for $a \neq z$.
Up to an inner automorphism, we can assume that $Z_0 \not\subseteq D$.
Then vectors $\z_i$ are fixed whenever $Z_i\not\subseteq D$.
First let's assume $a = y_j$ for some $j$.
When $Z_i\subseteq D$, we have
$\rho_\pi(\pi^{y_j}_D) (e_{z_i}) = -e_{y_j} +e_{z_i} + ge_{y_j}$.
We therefore get
$$\rho_\pi(\pi^{y_j}_D) (\z_i) = 
\begin{cases}
\z_i + 2 \y_j & \textrm{if $Z_i \subseteq D$,} \\
\z_i   & \textrm{if $Z_i \not\subseteq D$.}
\end{cases}$$
If instead $a = z_j$ for some $j$, then 
$\rho_\pi(\pi^{z_j}_D) (e_{z_i}) = -ge_{z_j} +ge_{z_i} + e_{z_j}$ and
$$\rho_\pi(\pi^{z_j}_D) (\z_i) = 
\begin{cases}
\z_i - 2 \z_j & \textrm{if $Z_i \subseteq D$,} \\
\z_i   & \textrm{if $Z_i \not\subseteq D$.}
\end{cases}$$

Another way to write this is by using the standard inner product on $\Q^s\cong W$.
Define $\w_D \in W$ as in Lemma~\ref{lem:span principal}: the $i$--th entry is equal to $1$ if $Z_i \subseteq D$, and $0$ otherwise.
Then
$$\rho_\pi(\pi^{y_j}_D) (\z_i) = 
\z_i + 2\langle \z_i , \w_D \rangle \y_j \ \ \textrm{ and } \ \ \rho_\pi(\pi^{z_j}_D) (\z_i) = 
\z_i - 2\langle \z_i , \w_D \rangle \z_j.
$$
This extends linearly over $W$, as per the following lemma.

\begin{lem}\label{lem:inner product}
	For $\v \in W$, and partial conjugations $\pi^{y_j}_D$ with $Z_0 \not\subseteq D$ and $j=1,\ldots,k$, we have  
	$$\rho_\pi(\pi^{y_j}_D)(\v) = \v + 2\langle \v,\w_D \rangle \y_j,$$
	$$\rho_\pi(\pi^{y_j}_D)(\x) = \x, \ \ \rho_\pi(\pi^{y_j}_D)(\y_l) = \y_l, \textrm{ for $l = 1,\ldots,k$.}$$
	For partial conjugations $\pi^z_D$ with $Z_0 \not\subseteq D$ and $z\in Z_i$, for $i=1,\ldots, s$, we have  
	$$\rho_\pi(\pi^z_D)(\v) = \v - 2\langle \v,\w_D \rangle \z_i,$$
	$$\rho_\pi(\pi^z_D)(\x) = \pm\x, \ \ \rho_\pi(\pi^z_D)(\y_l) = \pm\y_l \textrm{ for $l = 1,\ldots,k$.}.$$
	Meanwhile, any partial conjugation with multiplier $z$ in $Z_0$ acts as $-1$ on $\z_i$ if $Z_i \subseteq D$. Hence
	$$\rho_\pi(\pi^z_D) = \Id.$$
\end{lem}

\begin{proof}
	The statements regarding partial conjugations with multipliers $y = y_j$ or $z\in Z_i$ for $i>0$ follow from the discussion preceding the lemma, except when acting on $\x$ or $\y_j$. 
	With multiplier $y_j$ the action on $\x$ and $\y_l$ is trivial since $x$ and $y_l$ commute with $y_j$.
	With multiplier $z$, either $\cal{Y}$ is not in $D$ and the vectors $\x,\y_1,\ldots,\y_k$ are fixed, or $e_x$ (respectively $e_{y_l}$) is sent to $ge_x$ (respectively $ge_{y_l}$), meaning $\x \mapsto -\x$ (respectively $\y_j \mapsto -\y_l$).
	
	The final statement concerning those with multiplier $z\in Z_0$ follows from direct calculation.
	If $z_i \in D$, then 
	$$\pi^z_D(z_i) = z^{-1}z_i z 
	\implies 
	e_{z_i} \mapsto -ge_{z_0} + g e_{z_i} + e_{z_0} 
	\implies 
	\z_i \mapsto -\z_i;$$
	if $x \in D$, then
	$$\pi^z_D(x) = z^{-1}x z 
	\implies 
	e_{x} \mapsto -ge_{z_0} + g e_{x} + ge_{z_0} = ge_x
	\implies 
	\x \mapsto -\x.$$
	Calculations for any $y_i\in D$ are the same as for $x$.
\end{proof}

\section{Principal equivalence class-component pairs}\label{sec:principal}

We will introduce a new notion that concerns partial conjugations, 
with the objective being to refine the set of partial conjugations necessary to virtually generate the Torelli subgroup $\IA_\Gamma$ when $\Gamma$ has no SIL (see Lemma~\ref{lem:virtual torelli}).

Before we do so, we introduce some terminology.
Deleting the star of a vertex $x$ can divide $\Gamma$ up into multiple pieces, each the support of a partial conjugation with multiplier $x$.
We use the following to refer to these pieces.

\begin{defn}
	Let $x$ be a vertex of $\Gamma$. We use the phrase \emph{$x$--component} to mean a connected component of $\Gamma\setminus \st(x)$.
	
	Let $X$ be a subset of $\Gamma$. An \emph{$X$--component} is an $x$--component $C$ for some $x\in X$.
\end{defn}

Note that the terminology is similar, but different, to that found in \cite{CV11}, where a $\hat{v}$--component consists of a set of vertices in $\Gamma\setminus \st(v)$ that can be connected by a path that has no edge in $\st(v)$.

We will be focussing on $X$--components when $X$ is an equivalence class. There are two possibilities here depending on whether $X$ is abelian or not.
If it is, then each $X$--component is also an $x$--component for each $x\in X$.
If $X$ is a free equivalence class and $C$ is an $X$--component that does not intersect $X$, then $C$ is again an $x$--component for each $x\in X$.
However this is not always the case, the (only) exception to this rule being the following. If $x,x'\in X$, then $\{x'\}$ is an $x$--component, though clearly not an $x'$--component. 
In particular, each vertex in a free equivalence class $X$ forms an $X$--component.

If $x\geq y$, then by Lemma~\ref{lem:dominating multiplier in pc}, the $x$--components are each a union of $y$--components, with vertices from $\st(x)$ removed.
We are interested in when a partial conjugation $\pi^x_C$ can be expressed as a product in $\Out(A_\Gamma)$ of partial conjugations whose supports are $y$--components, for $y\leq x$.

\begin{defn}\label{def:virtuall obtained from dominated components}
	Let $x$ be a vertex of $\Gamma$ and $C$ an $x$--component.
	We say the partial conjugation  $\pi^x_C$ is \emph{virtually obtained from dominated components} if
	there exists $n \ne 0$ so that in $\Out(A_\Gamma)$
	$$
	(\pi^x_C)^n = (\pi^x_{D_1})^{\varepsilon_1} \cdots (\pi^x_{D_m})^{\varepsilon_2}
	$$
	where for each $i$ the set $D_i$ is a $y_i$--component for some $y_i$ dominated by, but not equal to $x$.
	
	If furthermore each $D_i$ can be taken to be a $y_i$--component for some $y_i$ dominated by, but not \emph{equivalent} to $x$, then we say 
	$\pi^x_C$ is \emph{virtually obtained from dominated non-equivalent components}.
\end{defn}

The product in Definition~\ref{def:virtuall obtained from dominated components} is considered within $\Out(A_\Gamma)$. We note that in certain cases it is equivalent to look at products within $\Aut(A_\Gamma)$, however this is not always the case. 
Indeed, provided there is some vertex $y$ (respectively equivalence class $Y$) dominated by, and distinct from, the vertex $x$ (respectively equivalence class $X$), the inner automorphism by $x$ is itself virtually obtained from dominated components (respectively from dominated non-equivalent components).
However, if $\Gamma\setminus\st(x)$ is connected and $x$ is not dominated by any other vertex, then $\pi^x_C$ is inner and so is virtually obtained from dominated components (in $\Out(A_\Gamma)$, with the empty word).

Note that, the product being in $\Out(A_\Gamma)$ means $\pi^x_C$ (or its power) can be constructed using the complement of $C$ in $\Gamma\setminus \st(x)$.

\begin{defn}\label{def:principal}
	Let $X$ be an equivalence class in $\Gamma$ and $C$ an $X$--component.
	We say the pair $(X,C)$ is \emph{non-principal}
	if $\pi^x_C$ is virtually obtained from dominated non-equivalent components for any $x\in X$.
	Otherwise, we say $(X,C)$ is \emph{principal}.
\end{defn}

We introduce some notation here.
Let $Y$ be an equivalence class of $\Gamma$ and $C$ a $Y$--component.
Define the set $P^Y_C$ to be
$$P^Y_C = \{\pi^x_{C'} \mid x \geq Y, C' = C \setminus \st(x)\}.$$
Then $\pi^x_C$ is virtually obtained from dominated non-equivalent components if there is a non-zero integer $n$ such that
\begin{equation}\label{eq:principal}
(\pi^x_C)^n \in \grp{P^Y_D \mid Y\leq X, Y\neq X, \textrm{ and $D$ a $Y$--component}}
\end{equation}
for each $x\in X$.

Observe that if for any $x\in X$ there is only one $x$--component, then $\pi^x_C$ is inner and hence is virtually obtained from dominated non-equivalent components.
Thus, if $(X,C)$ is principal then $\Gamma\setminus\st(x)$ is not connected for any $x\in X$.

An immediate example of principal pairs $(X,C)$ occur when $X$ is a domination-minimal equivalence class such that for any $x\in X$ the star of $x$ separates $\Gamma$.
A simple non-example is given  in Example~\ref{ex:nonprincipal} below. 
An example where we have to take $n>1$ is given in Example~\ref{ex:nonobv-nonprincipal}.
Further examples are given in Section~\ref{sec:virt ind} below, in Examples~\ref{ex-case1}, \ref{ex-case3}, and~\ref{ex-case2}.

\begin{ex}\label{ex:nonprincipal}
	We consider the graph in Figure~\ref{fig:ex-nonprincipal}.
	In this graph, the vertex $x$ is dominated by vertices $y,z_1,z_2$ and $z_3$, and forms its own equivalence class $X=\{x\}$.
	The $X$--components are $C_1 = \{z_1\}$, $C_2 = \{z_2\}$, and $C_3 = \{z_3\}$.
	The pairs  $(X,C_1)$ and $(X,C_2)$ are non-principal since $C_1$ and $C_2$ are also  $z_3$--components,
	meaning that $\pi^x_{C_1} \in P^{[z_3]}_{C_1}$ and $\pi^x_{C_2} \in P^{[z_3]}_{C_2}$.
	Since $(\pi^x_{C_3}) ^{-1} = \pi^x_{C_1} \pi^x_{C_2}$ in $\Out(A_\Gamma)$, we also deduce that $(X,C_3)$ is non-principal as $\pi^x_{C_3} \in \langle P^{[z_3]}_{C_1},P^{[z_3]}_{C_2}\rangle $.

	\begin{figure}[h!]
		\centering
		\labellist
		\pinlabel $x$ at 95 162
		\pinlabel $z_3$ at 95 0
		\pinlabel $z_1$ at -8 83
		\pinlabel $z_2$ at 225 83
		\pinlabel $y$ at 183 162
		\endlabellist
		\includegraphics[width=7cm]{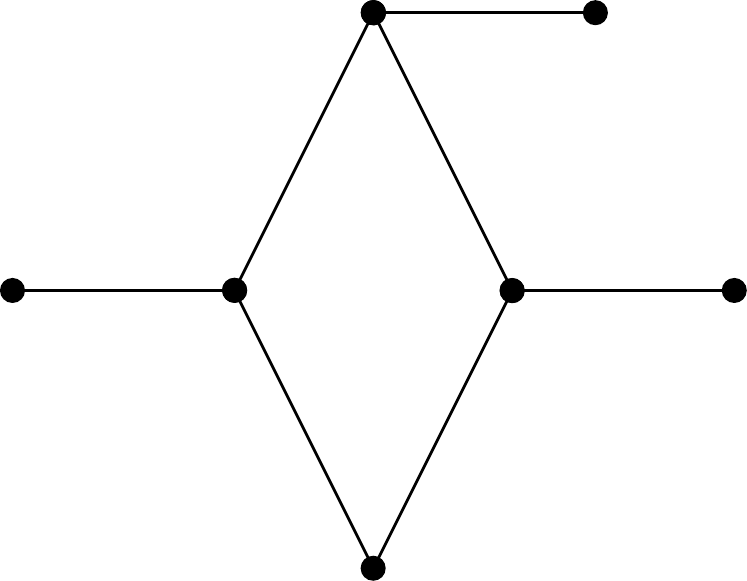}
		\caption{		
		}		\label{fig:ex-nonprincipal}
	\end{figure}
	
\end{ex}

\begin{ex}\label{ex:nonobv-nonprincipal}
	We consider the graph $\Gamma$ constructed using Figure~\ref{fig:ex-nonob-nonprincipal}.
	There are four $x$--components, $C_0,C_1,C_2$, and $C_3$, labelled so that $z_i \in C_i$.
	The only vertices dominated by $x$ are those added to the diagram: $y_1$, $y_2$, and $y_3$.
	The $y_1$--components are $C_0 \cup C_3$ and $C_1\cup C_2$;
	the $y_2$--components are $C_0 \cup C_1$ and $C_2 \cup C_3$;
	the $y_3$--components are $C_0 \cup C_2$ and $C_1 \cup C_3$.
	Then 
	$$\pi^x_{C_0 \cup C_3} \pi^x_{C_0 \cup C_1} \pi^x_{C_0 \cup C_2} = (\pi^x_{C_0})^{2} \ad_x = (\pi^x_{C_0})^{2}$$
	implying
	$$(\pi^x_{C_0})^{2} \in \left\langle P^{[y_1]}_{C_0 \cup C_3} , P^{[y_2]}_{C_0 \cup C_1} , P^{[y_3]}_{C_0 \cup C_2}\right\rangle.$$
	
	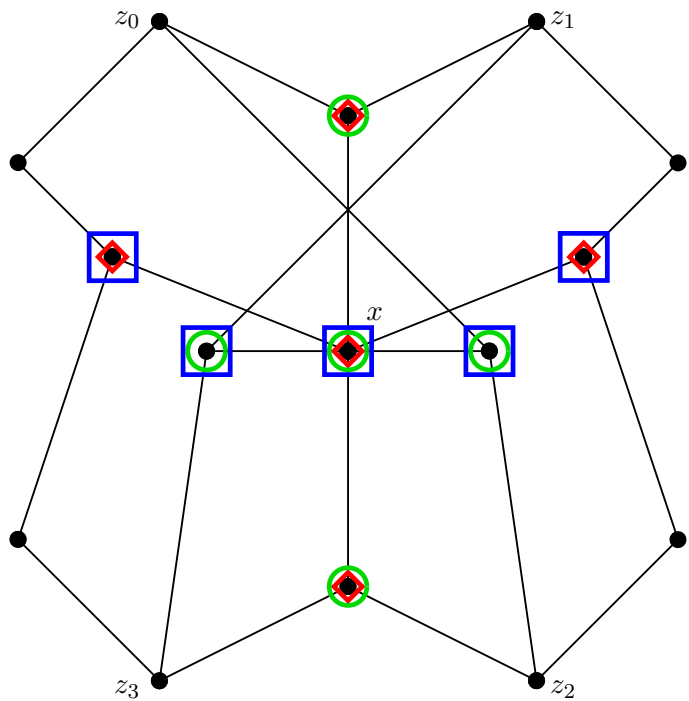
\begin{figure}[h!]
		\centering
		\labellist
		\pinlabel $x$ at 310 320
		\pinlabel $z_3$ at 100 0
		\pinlabel $z_1$ at 470 570
		\pinlabel $z_2$ at 470 0
		\pinlabel $z_0$ at 100 570
		\endlabellist
		\includegraphics[width=9cm]{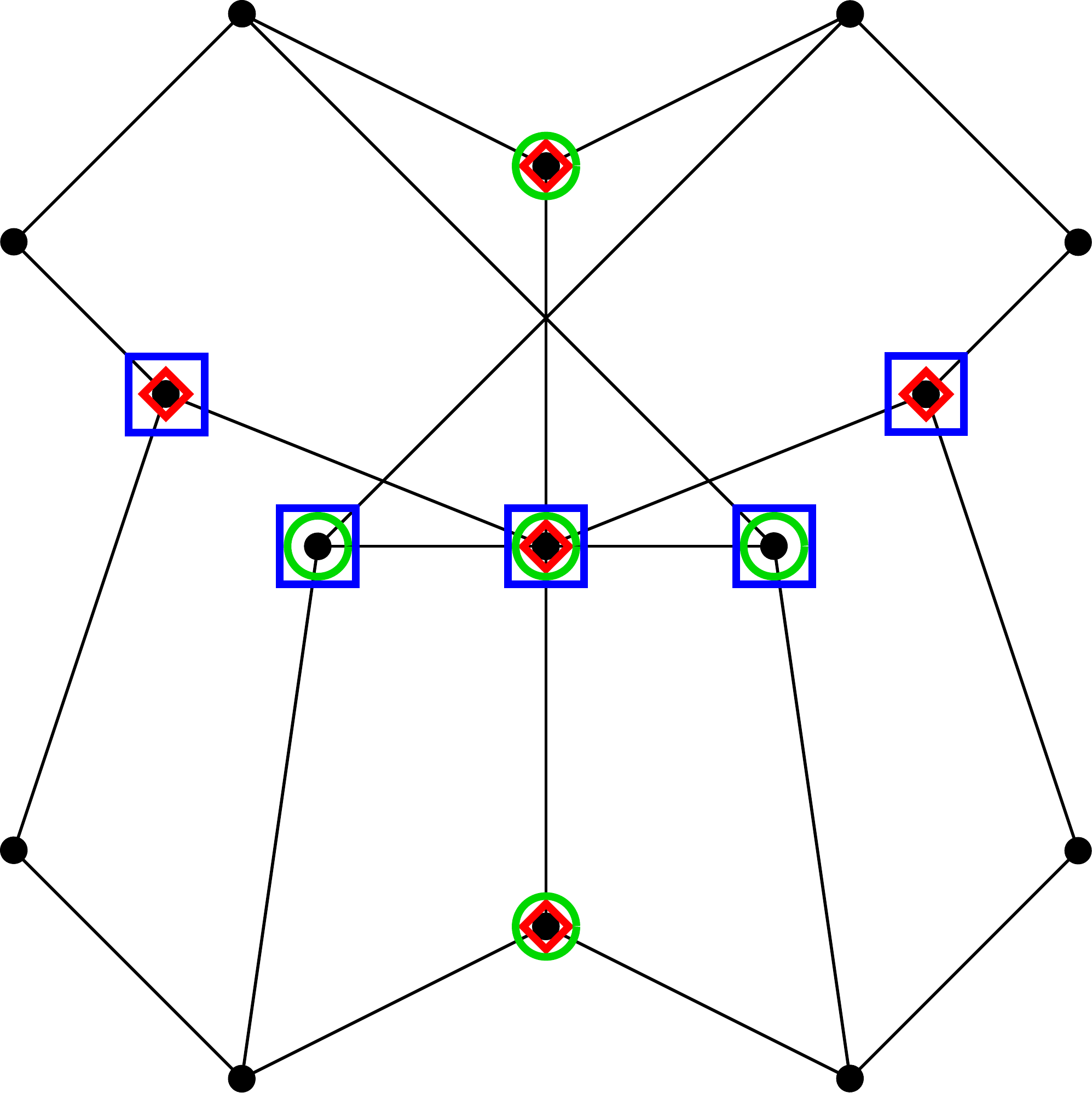}
		\caption{To construct $\Gamma$, add three vertices to the graph shown here.
			Add a vertex $y_1$ whose link consists of the vertices with a green circle \textcolor{green}{$\bigcirc$};
			Add a vertex $y_2$ whose link consists of the vertices with a blue square \textcolor{blue}{$\square$};
			Add	a vertex $y_3$ whose link consists of the vertices with a red diamond \textcolor{red}{$\diamond$}.
		}		\label{fig:ex-nonob-nonprincipal}
	\end{figure}	
\end{ex}

The following lemma translates the problem of determining principality into a linear algebra question.
We will exploit this approach in Section~\ref{sec:virt ind image}.

\begin{lem}\label{lem:span principal}
	Let $X$ be an equivalence class and let $C_0,C_1,\ldots, C_r$ be the $X$--components.
	Let $W$ be an $r$--dimensional vector space over $\Q$, and let $\{\z_1,\ldots,\z_r\}$ be a basis.
	For each vertex $y$ dominated by, but not in, $X$, and each $y$--component $D$, let $\w_D$ denote the vector 
	$$\w_D = w_1 \z_1 + \cdots + w_r \z_r$$
	where 
	$$w_i = \begin{cases}
	1 & \textrm{if $C_i \subset D$,} \\
	0 & \textrm{otherwise.}
	\end{cases}
	$$
	Let $\Delta$ denote the set of all vectors $\w_D$ as $D$ varies among all $y$--components, for all vertices $y$ dominated by, but not in, $X$.
	
	Then
	the vector $\langle 1,1,\ldots,1\rangle $ is in $\Span_\Q(\Delta)$ if and only if $(X,C_0)$ is non-principal.
\end{lem}

\begin{proof}
	Consider a product of partial conjugations with multiplier $x \in X$ in the sets $P^{Y}_{D}$, for equivalence classes $Y$ that are dominated by, but not equal to, $X$, and $Y$--components $D$.
	We write the product as
	$$
	(\pi^x_{D_1})^{\varepsilon_1} \cdots (\pi^x_{D_l})^{\varepsilon_l}
	$$
	where each $D_i$ is equal to a $Y$--component with $\st(x)$ removed, and $\varepsilon_i = \pm 1$.
	Up to an inner automorphism, and flipping the sign of $\varepsilon_i$, we may assume $C_0 \not\subset D_i$, and we then take the vector $\w_{D_i}\in\Delta$ defined above.
	
	The vector $ \w =  \sum \varepsilon_i \w_{D_i}$ records the action of $x$ on the vertices of each component $C_j$ in the following way.
	For $z\in C_1\cup\cdots\cup C_r$, we have 
	\begin{equation}\label{eq:product of pcs}
	(\pi^x_{D_1})^{\varepsilon_1} \cdots (\pi^x_{D_l})^{\varepsilon_l} (z)
	=
	x^{-\alpha_j} z x^{\alpha_j}
	\end{equation}
	where 
	$$\alpha_j 
	= \sum\limits_{\{ i \mid C_j \subset D_i\}} \varepsilon_i
	\ \
	\textrm{ and }
	\ \
	\langle \alpha_1,\ldots,\alpha_r \rangle = \sum\limits_{i=1}^l \varepsilon_i \w_{D_i}.$$
	If $(X,C_0)$ is not principal, then $(\pi^x_{C_0})^n$, for some integer $n$, can be written as a product as above.
	In particular, we would have $\alpha_i = n$ for each $i$, implying $\langle 1,1,\ldots,1\rangle$ is in the span of $\Delta$.
	
	On the other hand, if we can write $\langle n,n,\ldots, n\rangle $ as a $\Z$--linear combination of the vectors $\Delta$, then by the above we can translate that into a product of partial conjugations from sets $P^Y_D$, with $Y\leq X$ and $Y\ne X$, that equal $(\pi^x_{C_0})^n$.
\end{proof}

The following lemma explains why we refer to the pairs as principal or non-principal.
This is important in Section~\ref{sec:T} when we deal with graphs with no SIL.
The Torelli subgroup is the kernel of the standard representation, obtained by acting on the abelianisation of $A_\Gamma$.
For this lemma, all we need to know about the Torelli subgroup is that when there is no SIL it is abelian and is generated by the set of partial conjugations (see also Proposition~\ref{prop:Torelli}).

\begin{lem}\label{lem:virtual torelli}
	When $\Gamma$ contains no SIL, the union of the sets $P^X_C$ for which $(X,C)$ is principal generates a finite index subgroup of the Torelli subgroup $\IA_\Gamma$.
\end{lem}

\begin{proof}
	Seeing that these sets generate a finite index subgroup is done by first noting that some power of every partial conjugation $\pi^y_D$ is the subgroup generated by the sets $P^X_C$ with $(X,C)$ principal.
	This follows from the definition of principal/non-principal pairs.
	Together with the further knowledge that $\IA_\Gamma$ is abelian and generated by the set of all partial conjugations, this proves the lemma.
	
	The fact that, when there is no SIL, the Torelli subgroup is generated by partial conjugations is well-known and follows from the generating set of Day \cite[\S 3]{DayPeakReduction} for $\IA_\Gamma$. 
	The generating set consists of all partial conjugations and commutator transvections $R_u^{[v,w]} = [R_u^v,R_u^w]$.
	For a commutator transvection to be non-trivial we require $u\leq v,w$, and all three vertices pair-wise non-adjacent.
	These conditions imply the existence of a SIL $(v,w\mids u)$ (see Lemma~\ref{lem:domination pairs}).
	
	The fact that the Torelli subgroup is abelian follows from the fact that partial conjugations commute when there is no SIL, see for example \cite[Theorem 1.4]{GPR_automorphisms}.
\end{proof}

\section{Virtual indicability from partial conjugations}\label{sec:virt ind}

We begin by generalising condition~\ref{B}, defining the following property of a graph $\Gamma$.
Note that Lemma~\ref{lem:B' simply} below gives an alternate definition with as much jargon removed as possible.

\begin{enumerate}
	\renewcommand{\theenumi}{\normalfont (A\arabic{enumi}$'$)}
	\renewcommand\labelenumi{\theenumi}
	\setcounter{enumi}{1}
	\item \label{B'} if an equivalence class $X$ admits an $X$--component $C$ for which $(X,C)$ is principal, then  $|X| > 1$.
\end{enumerate}

The aim of this section is to prove Theorem~\ref{thm:no T special}, that if $\Gamma$ fails condition~\ref{B'} then $\Aut(A_\Gamma)$ is virtually indicable.

Recall, $([x],C)$ being principal means that
$\pi^x_C$ cannot be virtually obtained from dominated non-equivalent components.
That is, no non-trivial power of $\pi^x_C$ can be expressed as a product of partial conjugations $\pi^x_D$, where $D$ is a $y$--component for some $y$ dominated by, but not equivalent to, $x$ (see Definition~\ref{def:virtuall obtained from dominated components}). 
The product takes place in $\Out(A_\Gamma)$.

Condition \ref{B} of Aramayona--Mart\'inez-P\'erez \cite{Aramayona-MartinezPerez} implies that the minimal star-separating equivalence classes have size greater than one. 
These will of course be principal, thus, as mentioned in the introduction, \ref{B'} implies \ref{B}.

We remark that the definition of \ref{B'} is a peculiar mix of graphical and algebraic conditions.
We can ask whether it is possible to express the condition in purely graphical language.
However, we feel somewhat pessimistic about hopes to answer this positively.
Indeed, Example~\ref{ex:nonobv-nonprincipal} gives a situation where some na\"ive criterion to deny principality---either $C$ is a $y$--component  for some $y \leq X$, $y\notin X$, or the complement of $C$ in $\Gamma\setminus (\st(x) \cup X)$  can be expressed as a disjoint union of such components---may apply.
Instead, we suspect that using something like Lemma~\ref{lem:span principal} is as close as we may get, replacing the group setting with linear algebra.

The following is the reformulation of condition~\ref{B'} given in the introduction, which removes the language of equivalence classes and principality.

\begin{lem}\label{lem:B' simply}
	A graph $\Gamma$ satisfies \ref{B'} if and only if for every vertex $x$ in $\Gamma$ and every $x$--component $C$, 
	the partial conjugation $\pi^x_C$ is virtually obtained from dominated components.
\end{lem}

\begin{proof}
	We prove the contrapositive statement.
	We have failure of \ref{B'} if and only if there is some equivalence class $X$ of size one and an $X$--component $C$ so that $(X,C)$ is principal.
	That is to say, $X=\{x\}$, and 
	$\pi^x_C$ is not virtually obtained from dominated non-equivalent components.
	Since no vertex is equivalent to $x$, this gives the ``if'' direction.
	
	Now suppose we have $x$ and an $x$--component $C$ so that $\pi^x_C$ is not virtually obtained from dominated components.
	We want to show that $[x]=\{x\}$ and $([x],C)$ is principal, so \ref{B'} fails.
	The latter follows immediately once we have shown the former.
	So suppose $x \sim x'$.
	Then we necessarily have $x' \in C$ since otherwise $C$ would be an $x'$--component, contradicting the hypothesis on $\pi^x_C$.
	Since $C$ is an $x$--component, we must furthermore have $C = \{x'\}$ and $[x] = \{x,x'\}$.
	Suppose $C_1,\ldots, C_r$ are the remaining $x$--components.
	Then each $C_i$ is also an $x'$--component
	and
	$(\pi^x_C)^{-1} = \pi^x_{C_1} \cdots \pi^x_{C_r}$
	again contradicts our assumption on $\pi^x_C$.
	Hence no vertex is equivalent to $x$, and the lemma follows.
\end{proof}

\subsection{Outline of proof}

The objective is to prove Theorem~\ref{thm:no T special}, asserting that if condition~\ref{B'} fails, then $\Aut(A_\Gamma)$ is virtually indicable.
Thus we assume \ref{B'} fails, 
so there must be a vertex $x$, with $[x]=\{x\}$, and an $x$--component $C$ so that $(X,C)$ is principal.
We choose a domination-minimal such $x$ and will construct a homomorphism from a finite index subgroup of $\Out(A_\Gamma)$ onto $\Z$.
The process to construct such a virtual map onto $\Z$  involves refining the graph $\Gamma$ through projection and restriction maps.
We note that the domination relation that we will be referring to throughout always refers to domination in $\Gamma$.

\begin{step}
	Decluttering $\Gamma$.
\end{step}

\noindent We aim to exploit the partial conjugation $\pi^x_{C_0}$ to obtain a surjection onto $\Z$. The $x$--components therefore play a crucial role.
The first step, completed in Section~\ref{sec:decluttering}, is to remove much of $\Gamma$ by using a projection map,
but being sure to leave something in each $x$--component for the partial conjugations by $x$ to act on.

If $C=C_0,C_1,\ldots,C_r$ are the $x$--components, we explain in Definition~\ref{def:hat Lambda} how to choose a subset $Z_i$ of $C_i$, for each $i$, to create a subgraph $\hat{\Lambda}$ of $\Gamma$ that admits a projection map
$$
p_1 \colon \Sout^0(A_\Gamma) \to \Out(A_{\hat{\Lambda}}).$$
The construction of $\hat{\Lambda}$ ensures it contains $x$ and all vertices dominated by $x$.

The graph $\hat{\Lambda}$ is disconnected.
In Lemma~\ref{lem:abelianise Lambda hat} we verify that we can abelianise each component to obtain a graph $\Lambda_0$, 
so that $A_{\Lambda_0}$ is a free product of free abelian groups,
and which furthermore admits a homomorphism from the image of $p_1$ to $\Out(A_{\Lambda_0})$. 
Composing this with $p_1$ we get a homomorphism
$$
p \colon \Sout^0(A_\Gamma) \to \Out(A_{\Lambda_0}).$$

\begin{step}
	Dividing into cases.
\end{step}

\noindent The arrangement of the subsets $Z_i$ with regards to $x$ and the vertices it dominates can cause different issues to arise. 
We use these to split into two cases in Section~\ref{sec:virt ind cases}.
In the first case we ultimately yield a homomorphism 
$$
q \colon \Sout^0(A_\Gamma) \to \Out(\Z^{c_0} \ast \Z^{c_1} \ast \Z)$$
for some positive integers $c_0$ and $c_1$.
In the second case we get a similar map, but we may have more than three free factors in the target group.
We get
$$
q \colon \Sout^0(A_\Gamma) \to \Out(\Z^{c_0} \ast \Z^{c_1}\ast \cdots \Z^{c_s} \ast \Z^d)$$
for positive integers $c_0,\ldots, c_r,d$.

\begin{step}
	Employing the homological representation of Section~\ref{sec:VI homo rep}.
\end{step}

\noindent Readers familiar with homological representations will appreciate that  in the first  case, 
if the image of $q$ in $\Out(\Z^{c_0} \ast \Z^{c_1} \ast \Z)$ is sufficiently rich (even just containing a non-trivial partial conjugation), 
then obtaining a virtual surjection onto $\Z$ is entirely plausible.
In fact, if there are enough partial conjugations in the image, we can even obtain largeness of $\Out(A_\Gamma)$ (see Remark~\ref{rem:largeness}).

The second case though is more subtle.
The homological representation we use initially yields an action on a vector space of dimension $d+s$.
But by studying the action of partial conjugations and transvections on this space
we are able to restrict to a subspace $V$ so that the image of a finite-index subgroup of $\Aut(A_\Gamma)$ in $\PGL(V)$ is free abelian.

\subsection{Projecting to a simpler RAAG}\label{sec:decluttering}

In this section we do some spring cleaning on $\Gamma$, removing as much as we can while leaving enough so that the partial conjugations with multiplier $x$ still act non-trivially.
To do this, we pick out a minimal equivalence class in each $x$--component to keep, and kill the rest of the component through a projection map, as described in Section~\ref{sec:prelim:restriction projection}.

\begin{defn}\label{def:hat Lambda}
	The following describes how to construct the graph $\hat{\Lambda}$.	
	\begin{itemize}
		\item Let $C=C_0,C_1,\ldots,C_r$ be the $x$--components.
		If $C_i$ consists of one vertex, set $Z_i = C_i$.
		Otherwise, let $Z_i$ be a domination-minimal equivalence class in $C_i$.
		
		\item 	Let $Y$ be the set of vertices in $\st(x)$ that are dominated by $x$.
		
		\item 	Define $\hat{\Lambda}$ to be the induced subgraph of $\Lambda$ with vertex set
		$Z_0 \cup \cdots \cup Z_r \cup Y \cup \{x\}$.
	\end{itemize}
\end{defn}

To see that the sets $Z_i$ are well-defined, we observe that if $C_i$ contains more than one vertex, then the equivalence class of each vertex in $C_i$ is contained in $C_i$.
Indeed, if $z \in C_i$ then $|C_i|>1$ implies $\lk(z)\cap C_i$ is nonempty.
Suppose $z'$ is equivalent to $z$.
Firstly, this implies $z'$ cannot be in $\st(x)$.
Next, if $z'$ is adjacent to $z$ then $z'$ is immediately in $C_i$.
Otherwise $\lk(z) = \lk(z')$, so the fact that $\lk(z)\cap C_i$ is nonempty gives a two-edge path in $C_i$ connecting $z$ to $z'$.

This observation also implies that a set $Z_i$ can fail to be a complete equivalence class only if $Z_i = C_i$ consists of one vertex that is dominated by $x$ (this is addressed again in Lemma~\ref{lem:free Zi}).

\begin{lem}\label{lem:abelianise Lambda hat}
	If $v\in \hat{\Lambda}$ and $u\leq v$, then  $u\in \hat{\Lambda}$. 
	Hence the projection map
	$$
	p_1\colon \Sout^0(A_\Gamma) \to \Out^0(A_{\hat{\Lambda}})$$
	is well-defined.
\end{lem}

\begin{proof}
	First suppose $v\in Y\cup \{x\}$.
	Then $u\leq x$.
	Either $u\in Y$ or $\lk(u) \subset \st(x)$, which implies $u$ forms its own $x$--component so there is some $j$ so that $\{u\} = C_j = Z_j$.
	In particular, $u\in \hat{\Lambda}$.
	
	Now suppose $v \in Z_i$ and $u \notin Z_i$.
	We must then have that $u$ is not in $C_i$, and must also be in a different connected component $C_j$ of $\Gamma\setminus\st(x)$, since the only other alternative would be for $u$ to be in $\st(x)$, which implies $x\in \lk(u)\subseteq \st(v)$, contradicting $v\in C_i$.
	Then, as $u$ and $v$ are separated by $\st(x)$, and $\lk(u)\subseteq\st(v)$, we yield $\lk(u)\subseteq \st(x)$.
	Thus, $u$ forms its own $x$--component, so as above $u\in \hat{\Lambda}$.
\end{proof}

The structure of $A_{\hat{\Lambda}}$ as a group is as follows.
The subgroup generated by $Y$ could be any RAAG, denote it $A_Y$.
The $Z_i$ that are equivalence classes could be free or free abelian. 
They cannot be adjacent to any $y\in Y$ or $x$, since the link of $y$ is contained in the star of $x$.
Thus the group generated by $\hat{\Lambda}$ has the form
$$A_{\hat{\Lambda}} = G_0 \ast G_1  \ast \cdots  \ast G_r \ast (A_Y \times \Z)$$
where each $G_i$ is either $\Z^{c_i}$ or $F_{c_i}$ according to whether $Z_i$ is abelian or not, and where $c_i=|Z_i|$.
The final $\Z$ factor has generator $x$.

Following this projection $p_1$ we compose it with another homomorphism obtained by
abelianising each free factor in the above decomposition of $A_{\hat{\Lambda}}$.

\begin{defn}
	Let ${\Lambda_0}$ be the graph obtained from $\hat{\Lambda}$ by adding edges so that $A_Y$ and each equivalence class $Z_i$ become abelian.
	The structure of $A_{\Lambda_0}$ is therefore
	$$A_{{\Lambda_0}} = \Z^{c_0} \ast \Z^{c_1} \ast \cdots \ast \Z^{c_r} \ast \Z^{k+1}$$
	where $k = |Y|$.
\end{defn}

\begin{lem}
	The projection $p_1$ composes with a homomorphism $\im p_1 \to \Out(A_{\Lambda_0})$ to give a homomorphism
	$$p \colon \Sout^0(A_\Gamma) \to \Out(A_{\Lambda_0}).$$
\end{lem}

\begin{proof}
	The homomorphism sends an outer automorphism $\Phi$ in $\im p_1$ to the automorphism defined by the action of $\Phi$ on the vertices of $\Lambda_0$.	
	To see that $\Phi$ is indeed sent to an automorphism, one just needs to verify that the kernel of the map $A_{\hat{\Lambda}} \to A_{\Lambda_0}$, which is normally generated by the commutators $[a,b]$ for $a,b\in Z_i$ and $i=0,1,\ldots, r$, is preserved by each generator in the image of $p_1$.
	Indeed, any transvection $R_u^v$ preserves the kernel as either $u$ and $v$ are in the same factor $G_i$, or otherwise the factor containing $u$ must be cyclic (since $\lk(u)$ must therefore be contained in $\st(x)$, regardless of what $v$ is).
	Meanwhile, a partial conjugation $\pi^v_D$ with $v\in \hat{\Lambda}$ preserves the kernel since each factor in $A_{\hat{\Lambda}}$ either is contained in $D$, does not intersect $D$,
	or both $D$ and $v$ are contained in $Z_i$.
\end{proof}

We finish this subsection with a collection of useful observations about the behaviour of the sets $Z_i$.
Remember, the domination relation that we refer to (and hence the equivalence classes of vertices) always comes from the relation in $\Gamma$.

\begin{lem}\label{lem:z_i leq z_j implies z_i leq x}
	Suppose $z_i\in Z_i$ and $z_j \in Z_j$ for some $i\ne j$, and $z_i \leq z_j$.
	Then $z_i \leq x$, and $Z_i$ is a $z_j$--component.
\end{lem}

\begin{proof}
	Any path from $z_i$ to $z_j$ must hit $\st(x)$. Hence $\lk(z_i)=\lk(z_i) \cap \lk(z_j) \subseteq \st(x)$, so $z_i \leq x$.
	Also, $\lk(z_i) \subseteq \lk(z_j)$ implies that $Z_i$ is a $z_j$--component.
\end{proof}

\begin{lem}\label{lem:free Zi}
	An equivalence class is free in $\Lambda_0$ if and only if it is free in $\Gamma$, is dominated by $x$ in $\Gamma$, and is outside $\st(x)$.
	
	In particular, $Z_i$ and $Z_j$, for $i\ne j$ are in the same equivalence class in $\Gamma$ only if $Z_i,Z_j \leq x$, and in which case their equivalence class is not abelian and is contained in $Z_0\cup\cdots \cup Z_r$.
\end{lem}

\begin{proof}
The only way to have an equivalence class that remains free in $\Lambda_0$ is if it is the union of singleton sets $Z_i$.
Then Lemma~\ref{lem:z_i leq z_j implies z_i leq x} implies the class is dominated by $x$.

The last claim, that the equivalence class of $Z_i$ is contained in $Z_0\cup \cdots \cup Z_r$, is true  since each vertex $z$ of the equivalence class satisfies $\lk(z) \subset \st(x)$, and so forms an $x$--component on its own, and therefore must be one of the sets $Z_j$.
\end{proof}

\begin{lem}\label{lem:z cpnt is Ci}
	Let $z\in U = \{u\mid u\leq x, u\ne x\}$, $z\not\in\st(x)$, and $D$ be a $z$--component.
	Then either
	\begin{itemize}
		\item $D = C_i$ for some $i$, or
		\item $x\in D$.
	\end{itemize}
\end{lem}

\begin{proof}
	Suppose $x\notin D$, so $\st(z)$ separates $D$ and $x$.
	We claim this means that $D \subset \Gamma\setminus \st(x)$
	Indeed, otherwise $x$ would be adjacent to a vertex of $D$, and since $x\not\in\st(z)$ we would have $x\in D$.
	Thus, since $D$ is connected, it must be contained in an $x$--component $C_i$.
	Finally, since $z\leq x$ no vertex of $C_i$ is in $\st(z)$, and it therefore follows that $D=C_i$.
\end{proof}

The final two lemmas give us useful properties when we work under the assumption that condition~\ref{A} of \cite{Aramayona-MartinezPerez} holds.
When \ref{A} fails, Theorem~\ref{thms:AMP} from \cite{Aramayona-MartinezPerez} tells us that we get virtual indicability.

\begin{lem}\label{lem:Z0 not domianted by x}
	There is no $u\in U$ so that $C_0$ is a $u$--component.
	
	In particular, either $Z_0$ is not dominated by $x$, or condition~\ref{A} fails.
\end{lem}

\begin{proof}
	The fact that $C_0$ is not a $u$--component for any $u\in U$ follows immediately from principality of $(X,C_0)$.	
	If $Z_0$ is dominated by $x$, then $Z_0 = \{z_0\}$.
	If condition~\ref{A} holds then there is some other vertex $u$ such that $z_0 \leq u \leq x$.
	This implies that $Z_0=C_0$ is a $u$--component and $u\in U$, a contradiction.
\end{proof}

\begin{lem}\label{lem:Ci not u cpnt}
	There exists some $i\ne 0$ so that $C_i$ is not a $u$--component for any $u\in U$.
	
	In particular, either $Z_i$ is not dominated by $x$, or condition~\ref{A} fails..
\end{lem}

\begin{proof}
	If for each $C_i$ there was a $u_i\in U$ so that $C_i$ were a $u_i$--component, then
	$\pi^x_{C_i} \in P^{[u_i]}_{C_i}$.
	In particular, $(\pi^x_{C_0})^{-1} = \pi^x_{C_1} \cdots \pi^x_{C_r}$ is in the subgroup generated by  the sets $P^{[u_i]}_{C_i}$, for $i=1,\ldots, r$, contradicting principality of $(X,C_0)$.
	
	If $Z_i$ were dominated by $x$, then it cannot be dominated by any $u\in U$ as otherwise it would be a $u$--component.
\end{proof}

\subsection{Two cases}\label{sec:virt ind cases}

We divide into two cases  according to the possible arrangements of components $Z_i$.
Recall that $U = \{ u \mid u \leq x, u\ne x\}$.
The two cases are:

\begin{enumerate}[(1)]
	\item\label{case:same y} There is a class $Z_i$, not dominated by $x$, so that $Z_0$ and $Z_i$ are in the same $u$--component  for every $u\in U$
	\item\label{case:not same y} For every class $Z_i$, for $i\ne 0$ and which is not dominated by $x$, there is some $u\in U$ so that $Z_0$ and $Z_i$ are in separate $u$--components.
\end{enumerate}

\subsubsection{Examples}

We now give some examples of graphs satisfying each condition.
We will refer back to these examples throughout the remainder of the proof of Theorem~\ref{thm:no T special}.

Note that these are not ``new'' examples, in the sense that $\Aut(A_\Gamma)$ is already known to be virtually indicable by earlier results.
For example we can apply Theorem~\ref{thms:AMP} (a result of Aramayona and Mart\'inez-P\'erez \cite{Aramayona-MartinezPerez}) to obtain a virtual surjection onto $\Z$.
We use these examples though to highlight our method.
Some tricks, such as replacing vertices by equivalence classes, or introducing benign vertices, can be used to prevent the application of earlier results, but doing so overcomplicates matters and defeats the purpose of including some, hopefully, enlightening examples.

\begin{ex}\label{ex-case1}
	The graph $\Gamma$ given in  Figure~\ref{fig:ex-case1} satisfies the conditions of case~\ref{case:same y}.
	There are five $x$--components, $C_0,\ldots, C_4$, and in each a domination-minimal vertex $z_i$ has been selected. 	
	The vertex $z_0$ dominates $x$; while $z_2$ and $z_3$ are dominated by $x$, and so form the set $U$.
	
	We claim that $(X,C_0)$ and $(X,C_1)$ are both principal.
	There is only one $z_2$--component.
	There are two $z_3$--components, $C_4$ being one,  and the other containing $C_0\cup C_1\cup C_2\cup \{x\}$ and some unlabelled vertices from $\st(x)$.
	As there is no way to separate $C_0$ and $C_1$ by the star of either $z_2$ or $z_3$, both $(X,C_0)$ and $(X,C_1)$ must be principal.

	\begin{figure}[h!]
		\vspace{3mm}
		\centering
		\labellist
		\pinlabel $x$ at -10 165
		\pinlabel $z_4$ at 250 335
		\pinlabel $z_0$ at 175 195
		\pinlabel $z_1$ at 250 -10
		\pinlabel $z_2$ at 0 -10
		\pinlabel $z_3$ at 0 335
		\endlabellist
		\includegraphics[width=8cm]{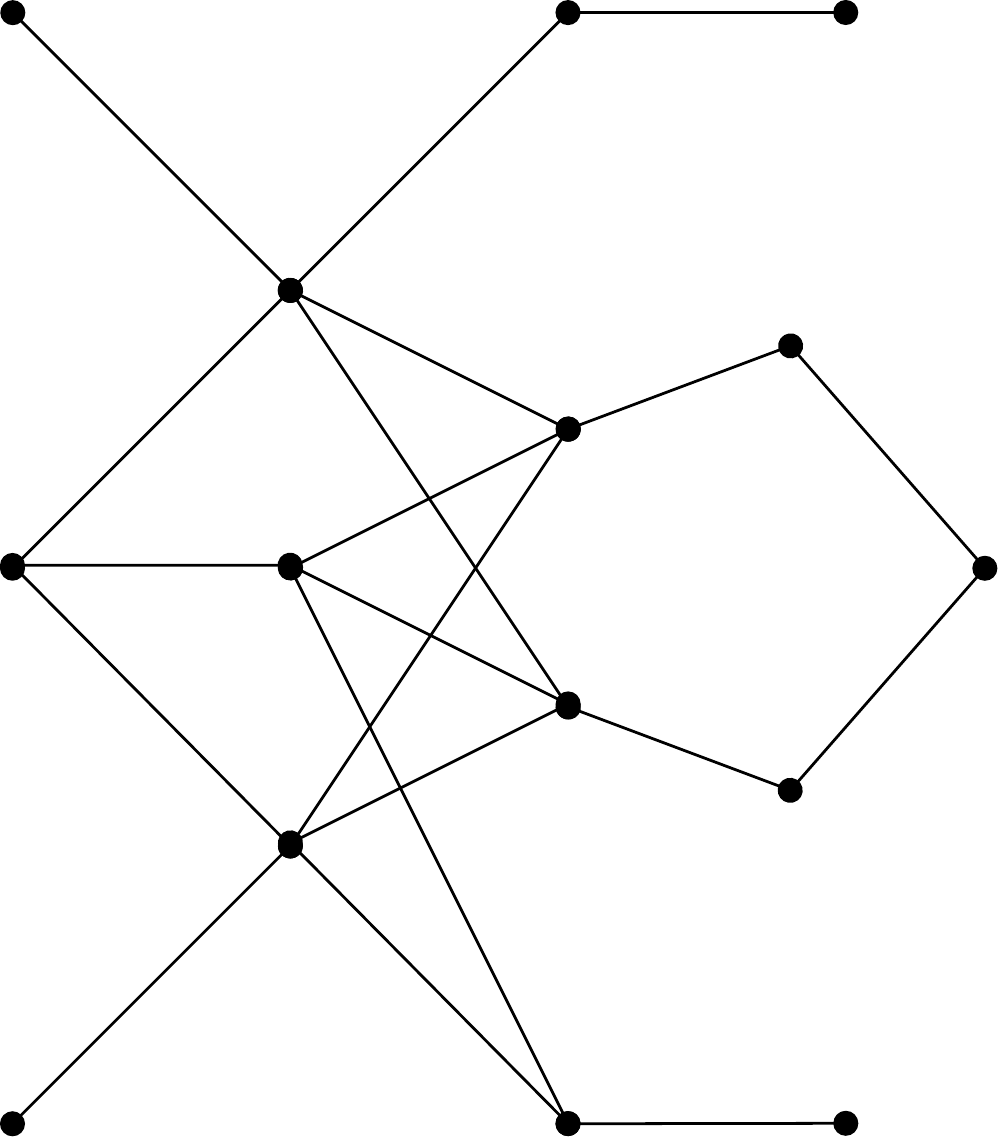}
		\vspace{3mm}
		\caption{The graph $\Gamma$ of Example~\ref{ex-case1} satisfying case~\ref{case:same y}.}
		\label{fig:ex-case1}
	\end{figure}
\end{ex}

\begin{ex}\label{ex-case3}
	The graph $\Gamma$ given in  Figure~\ref{fig:ex-case3} satisfies the conditions of case~\ref{case:not same y}.
		
	There are four $x$--components, which we can denote $C_i$ so that $z_i\in C_i$, for $i=0,1,2,3$.
	We can choose $Z_i = \{z_i\}$ for each $i$.
	
	We claim that $(X,C_0)$ is principal, and to show this we use Lemma~\ref{lem:span principal}
	Consider the vector space $W$ over $\Q$ with basis given by $\{\z_1,\z_2,\z_3\}$.
	The only vertices dominated by $x$ are $y_1$ and $y_2$. 
	For each partial conjugation $\pi^{y_i}_D$ we can, up to an inner automorphism, assume that $C_0\cap D = \emptyset$.
	Thus we have partial conjugations with supports  $C_1\cup C_2$ and $C_2\cup C_3$.
	These translate into vectors $\langle 1,1,0\rangle$ and $\langle 0,1,1\rangle$ as described in Lemma~\ref{lem:span principal}.
	It is clear that $\langle 1,1,1\rangle $ is not in the $\Q$--span of these two vectors, so $(X,C_0)$ is principal.

	\begin{figure}[h!]
		\vspace{3mm}
		\centering
		\labellist
		\pinlabel $x$ at 190 215
		\pinlabel $y_1$ at 122 105
		\pinlabel $y_2$ at 305 285
		\pinlabel $z_0$ at 88 418
		\pinlabel $z_1$ at 327 418
		\pinlabel $z_2$ at 327 -15
		\pinlabel $z_3$ at 88 -15
		\endlabellist
		\includegraphics[width=7.1cm]{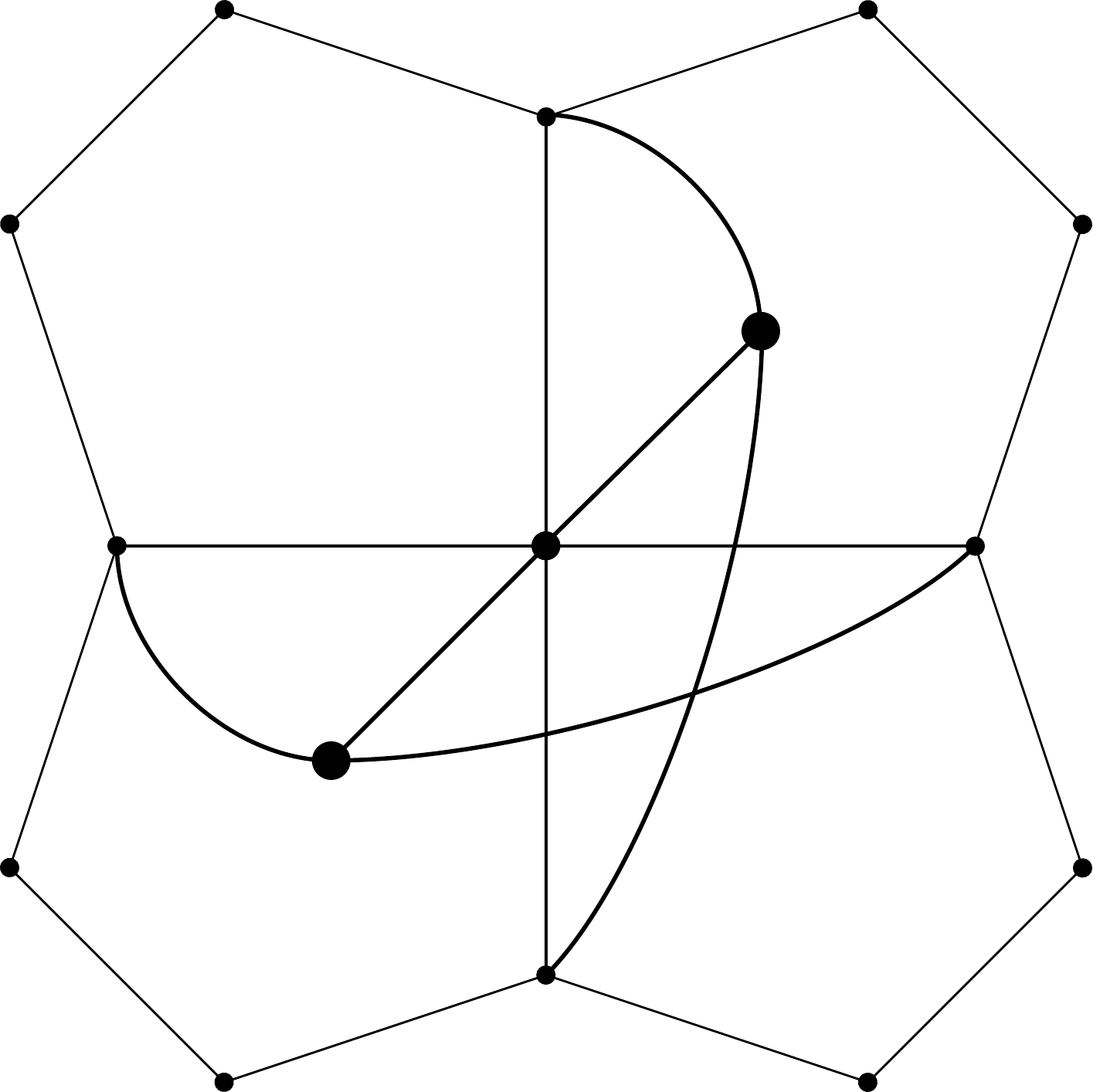}
		\vspace{3mm}
		\caption{The graph $\Gamma$ of Example~\ref{ex-case3} satisfying Case~\ref{case:not same y}.}
		\label{fig:ex-case3}
	\end{figure}
\end{ex}

\begin{ex}\label{ex-case2}
	The graph $\Gamma$ given in  Figure~\ref{fig:ex-case2} satisfies the conditions of case~\ref{case:not same y}.
	
	There are six $x$--components, $C_0,\ldots,C_5$, with a domination minimal vertex $z_i$ chosen in each $C_i$.
	The set $U$ consists of vertices $y_1$ and $y_2$.
	We note that $\st(y_1)$ separates $z_1,z_2$, and $z_3$ from $z_0$, while $\st(y_2)$ separates $z_3,z_4,$ and $z_5$ from $z_0$. This confirms we are in case~\ref{case:not same y}.
	
	To see that $(X,C_0)$ is principal, using the notation of Lemma~\ref{lem:span principal}, $\Delta$ consists of two vectors, $\langle 1,1,1,0,0\rangle $ from $y_1$, and $\langle 0,0,1,1,1\rangle $ from $y_2$.
	Evidently there is no way to get $\langle 1,1,1,1,1\rangle $ in the $\Q$--span of $\Delta$, implying $(X,C_0)$ is principal.
	
\begin{figure}[h!]
	\vspace{3mm}
	\centering
	\labellist
	\pinlabel $x$ at 165 299
	\pinlabel $y_1$ at 70 285
	\pinlabel $y_2$ at 259 285
	\pinlabel $z_0$ at 338 363
	\pinlabel $z_1$ at 245 458
	\pinlabel $z_3$ at -10 363
	\pinlabel $z_4$ at 165 110
	\pinlabel $z_5$ at 165 190
	\pinlabel $z_2$ at 165 363
	\endlabellist
	\includegraphics[width=8cm]{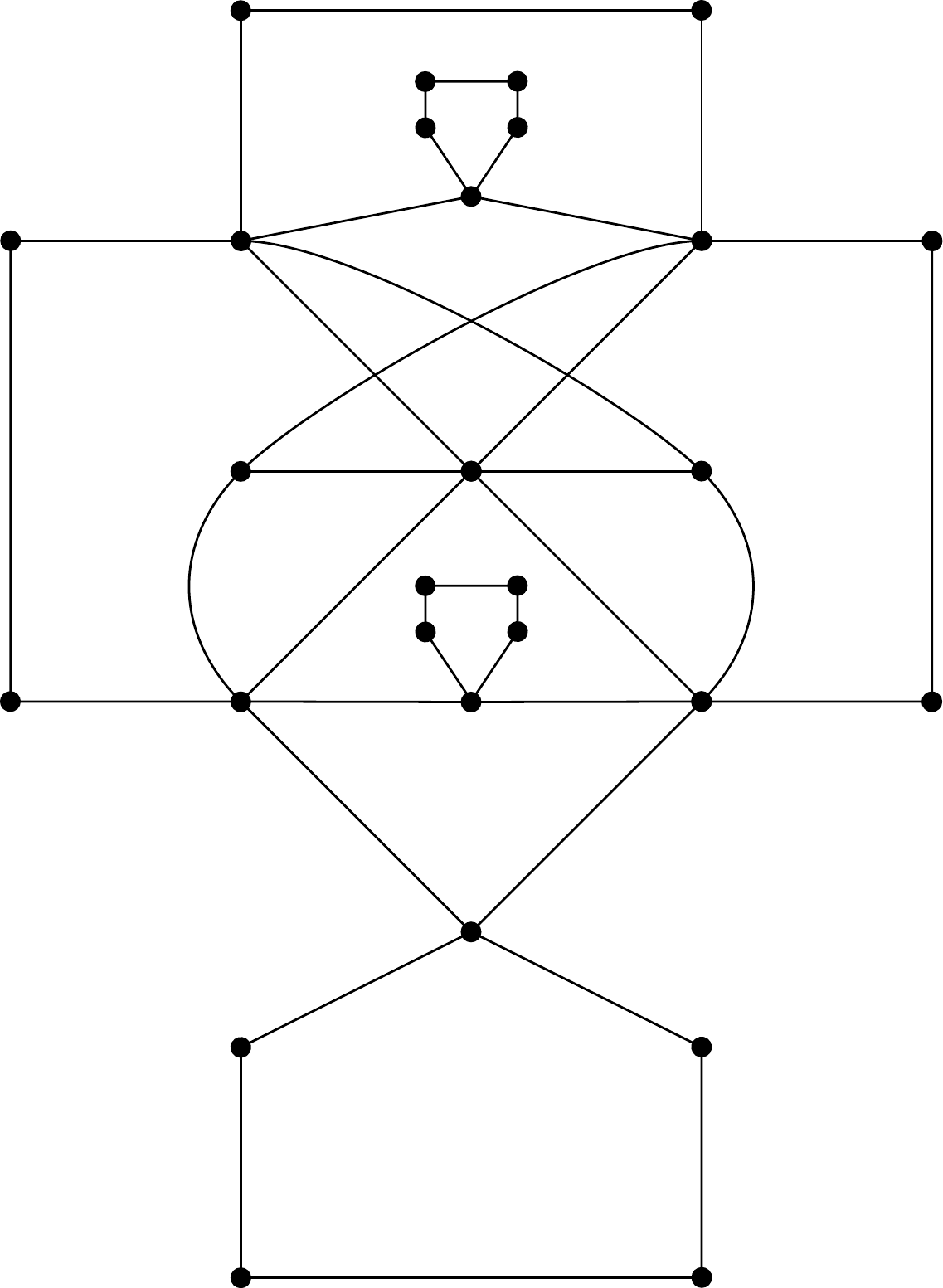}
	\vspace{3mm}
	\caption{The graph $\Gamma$ of Example~\ref{ex-case2} satisfying Case~\ref{case:not same y}.}
	\label{fig:ex-case2}
\end{figure}	
\end{ex}

\subsection{Refining the graph further}\label{sec:q}

We will describe two different ways to cut down $\Lambda_0$ to a smaller graph $\Lambda$, depending on which case we are in.
Ultimately, in each case we claim there are integers $c_0,c_1,\ldots,c_s,d$ and a homomorphism 
	$$q \colon \Sout^0(A_\Gamma) \to \Out(\Z^{c_0} \ast \cdots \ast \Z^{c_s} \ast \Z^d)$$
such that the image is sufficiently rich so that we can use a homological representation to get a virtual map to $\Z$.

Note that we will work under the assumption that conditions \ref{A} and \ref{B} hold in $\Gamma$, so Theorem~\ref{thms:AMP} does not already yield virtual indicability.
In particular, Lemmas~\ref{lem:Z0 not domianted by x} and~\ref{lem:Ci not u cpnt} tell us there is some $i>0$ so that $Z_0$ and $Z_i$ are not dominated by $x$.

\subsubsection{Case~\ref{case:same y}}

Relabel the classes $Z_i$ for $i>0$, if necessary, so that $Z_1, \ldots, Z_{t}$ are 
all the classes that are not dominated by $x$ and that are not separated from $Z_0$ by $\st(u)$, for any $u \in U$.
We first refine the graph $\Lambda_0$ by removing all classes $Z_i$ for $i>t$, and all vertices $y\in Y$, leaving a graph $\Lambda_1$ that has vertex set $Z_0 \cup \cdots \cup Z_t \cup \{x\}$.

First though we make the following observation.

\begin{lem}\label{lem:Zi dominates x implies case 1}
	If there is some $i$ such that $Z_i$ dominates $x$, then we are in case~\ref{case:same y} and, after the above relabelling, $i\leq t$.
\end{lem}

\begin{proof}
	Assume $Z_i$ dominates $x$.
	By Lemmas~\ref{lem:Z0 not domianted by x} (if $i>0$) and~\ref{lem:Ci not u cpnt} (if $i=0$), we know there is some $j\ne i$ so that $C_j$ is not a $u$--component for any $u\in U$, and furthermore that $j$ can be chosen  so that $0\in \{i,j\}$.
	Since  for any $u\in U$ we know that $C_j$ is not a $u$--component, it follows that $Z_j$ and $x$ are not separated by $\st(u)$.
	The same is true of $x$ and $Z_i$. Indeed, since $u$ is not equivalent to $x$, we may use a vertex that is in $\lk(x) \setminus \st(u)$ to connect $x$ to $Z_i$ outside of $\st(u)$.
	Hence, for each $u\in U$, we can connect $Z_i$ and $Z_j$ with a path via $x$ that avoids the star of $u$.
	It follows that we are in case~\ref{case:same y} and, after the above relabelling, we must also have $i\leq t$.
\end{proof}

\begin{lem}\label{lem:I restriction}
	The restriction map $r \colon \im p \to \Out(A_{\Lambda_1})$ is well-defined.
\end{lem}

\begin{proof}
	First consider transvections in the image of $p$.
	Suppose $z \in Z_i$ for some $i$ with $0\leq i \leq t$.
	By assumption $Z_1,\ldots,Z_t$ are not dominated by $x$, while $Z_0$ is not dominated by $x$ by virtue of Lemma~\ref{lem:Z0 not domianted by x}.
	Thus we cannot have $z\leq z'$ with $z' \in Z_{i'}$ and $i' \ne i$, since then Lemma~\ref{lem:z_i leq z_j implies z_i leq x} implies $Z_i$ would be dominated by $x$.	
	Similarly, we cannot have $z \leq y\in Y$.	
	Thus any transvection $R_z^w$ must have $w \in \Lambda_1$, and thus preserve $A_{\Lambda_1}$.
	
	The remaining transvections to consider are those of the form $R_x^w$.
	But then by Lemma~\ref{lem:Zi dominates x implies case 1}, $w$ must be a vertex of some $Z_i \subset \Lambda_1$ and so $A_{\Lambda_1}$ us again preserved.
	All other transvections in $\im p$ act on a vertex not in $\Lambda_1$ and so fix $A_{\Lambda_1}$.
	
	Next consider partial conjugations.
	By construction those of the form $\pi^y_D$ with $y\in Y$ act trivially on $\Lambda_1$. 
	Indeed, $x\in \st(y)$ and the remaining vertices of $\Lambda_1$, namely $Z_0 \cup Z_1 \cup \cdots \cup Z_s$, are all in the same $y$--component.
	
	That leaves us to consider the partial conjugations $\pi^z_D$ for $z\in Z_i$ for some $i>t$. 
	We want to show that $\st(z)$ cannot separate $\Lambda_1$ into more than one component.
	For contradiction, 
	first suppose $\st(z)$ separates $Z_0$ and $Z_j$ for some $j\leq s$.
	Then every path from $Z_j$ to $Z_0$ passes through $\st(z)$, so $z\notin U$ by choice of $Z_j$.
	Furthermore, a path from $Z_j$ to $Z_0$ avoiding $\st(u)$, for $u\in U$ passes through $\st(z)$ and can therefore take a detour to $z$, yielding a path from $Z_0$ to $z$ avoiding $\st(u)$.
	Note that $z$ is not in $\lk(u)$ since $\lk(u)$ is contained in $\st(x)$.
	This implies that for each $u$ we have $z$ and $Z_0$ in the same component of $\Gamma\setminus \st(u)$.
	It follows that $z\in \Lambda_1$, a contradiction.
	
	On the other hand, suppose $\st(z)$ separates some $Z_j$ in $\Lambda_1$ and $x$.
	Since, as just proved, $\st(z)$ does not separate $Z_0$ and $Z_j$ for $j\ne 0$, we can assume $j\ne 0$.
	For each $u\in U$ there is a path $p$ from $Z_j$ to $Z_0$ avoiding $\st(u)$.
	Let $p'$ be the initial segment of $p$ starting in $Z_j$ and finishing at the first vertex of $p$ that is in $\st(x)$---such a vertex must exist, denote it by $w$.
	Then adjoin $x$ to the end of $p'$, creating a path $p''$ from $Z_j$ to $x$.
	Since, by assumption, $\st(z)$ separates $Z_j$ and $x$, some vertex of $p''$ must be in $\st(z)$.
	As $x \notin \st(z)$, such a vertex must be from $p'$, and hence $p$.
	It follows then that $z \ne u$, and therefore $z\notin U$.	
	As $z \notin Z_j$ then $\st(z) \cap C_j$ is empty and we must have $w \in \st(z)$.
	We can then append the vertex $w$ to the path $p'$, creating a path from $Z_j$ to $z$ that avoids $\st(u)$.
	As $Z_0$ and $Z_j$ are not separated by $\st(u)$ we get that $Z_0$ and $z$ are not either.
	As before, we get $z\in\Lambda_1$.
			
	Thus every transvection and partial conjugation in $\im p$ preserves $A_{\Lambda_1}$ and, since they generate the image, the restriction map is well-defined.
\end{proof}

We observe that when $Z_i$ is not dominated by $x$, as is the case for each $Z_1,\ldots, Z_t$, we can use a projection map to eliminate it.

\begin{lem}\label{lem:kill Z not dominated}
	Let $i_1,\ldots,i_n \in \{1,\ldots,t\}$.
	If $\Lambda$ is the graph obtained by removing $Z_{i_1},\ldots, Z_{i_n}$ from $\Lambda_1$,
	then the projection map
	$$p_2 \colon \im {r} \to \Out(A_{\Lambda})$$
	is well-defined.
\end{lem}

\begin{proof}
	For the projection map to be well defined we need only check the image of transvections (in $\im r$) map to automorphisms under $p_2$.
	In particular, the transvection $R_v^w$ will fail to map to an automorphism only if $v$ is not in  $\Lambda$ whereas $w$ is.
	We may therefore assume that $v$ is in some $Z_{i_j}$ that is deleted to make $\Lambda$.
	Since $v \not\leq x$, Lemma~\ref{lem:z_i leq z_j implies z_i leq x} tells us that $v$ is not dominated by any $z \in Z_i$ for $i \ne i_j$.
	So for $w$ to dominate $v$, it must therefore be in $Z_{i_j}$, hence not in $\Lambda$.
\end{proof}

We can therefore choose our favourite $i \in \{1,\ldots, t\}$ and use $p_2$ from Lemma~\ref{lem:kill Z not dominated} to get
$$q=p_2\circ r \circ p \colon \Sout^0(A_\Gamma) \to \Out(\Z^{c_0} \ast \Z^{c_i} \ast \Z).$$
We have $\Lambda$ the graph on vertex set $Z_0 \cup Z_i \cup \{x\}$.

\begin{rem}\label{rem:largeness}
If we can choose $Z_i$ here so that $\st(Z_i)$ separates $Z_0$ and $x$ in $\Gamma$, or so that $\st(Z_0)$ separates $Z_i$ and $x$, then in fact $\Aut(A_\Gamma)$ is large.
This is the case, for example, if $Z_0$ or $Z_i$ dominates $x$.
To obtain largeness, observe that the vertices of $\Lambda$ form a special SIL, as defined in \cite{GuirardelSale-vastness}, which implies largeness by \cite[Proposition~3.15]{GuirardelSale-vastness}.
The special SIL will be either $(x,z_0 \mids z_i)$ or $(x,z_i \mids z_0)$ where $z_i\in Z_i$ and $z_0\in Z_0$.
There is a technical note here though. 
Special SILs require the equivalence class of each vertex involved to be abelian. This will be so once we reach $\Lambda_0$ (by Lemma~\ref{lem:free Zi}), but is not necessarily the case in $\Gamma$ itself.
Then, to obtain largeness from the special SIL one needs to verify that the required SIL automorphisms survive the map $p$.
\end{rem}

\begin{ex}
	We continue Example~\ref{ex-case1}.
	The classes $z_0$ and $z_1$ are not separated by any star from $U$, while $z_4$ is separated from $z_0$ by $\st(z_3)$.
	Thus we take $t=1$, and restrict to $A_{\Lambda_1}$, with $\Lambda_1$ given by vertices $\{z_0,z_1,x\}$.
	The map $p_2$ is not required in this example.
\end{ex}

\subsubsection{Case~\ref{case:not same y}}

In this case, we use a projection map to kill the partial conjugations $\pi^z_D$ which would otherwise later cause trouble in the image of the homological representation.
We gather these partial conjugations up in a set $B$, defined to consist of partial conjugations $\pi^z_D \in \im p$ such that 
	\begin{itemize}
		\item $z \in Z_i$ for some $i\ne 0$,
		\item 	$z \not\leq x$,
		\item 	and $D$ is not a union of $u$--components for any $u\in U$
	\end{itemize}

The aim is to kill enough of the $Z_i$ so that the partial conjugations in $B$ act trivially on what remains, but not to kill too many so that the action of $\pi^x_{C_0}$ becomes trivial, and furthermore, so that we can still track the action of partial conjugations by $u \in U$ on each component $C_i$, whether it has been deleted or not.

The following lemma is a crucial tool in accomplishing this.

\begin{lem}\label{lem:bad supports in y-cpnt}
	Suppose $\pi^z_D \in B$.
	Then either
	\begin{itemize}
		\item $D$ and $z$ are in the same $u$--component for every $u \in U$, or
		\item $x\in D$.
	\end{itemize}	
\end{lem}

\begin{proof}
	Suppose $x\notin D$.
	In this case $D\cap \st(x) = \emptyset$
	and any path from $D$ to $x$ intersects $\st(z)$.
	For contradiction, suppose that $D$ and $z$ are separated by $\st(u)$, for some $u\in U$.
	Since $u\leq x$  and $z \not \leq x$, we must have that $u$ is in a different $x$--component to $z$.
	Therefore $D$ cannot intersect the $x$--component containing $z$, since $\st(u) \subseteq \{u\} \cup \st(x)$.
	As the star of $z$ cannot intersect any $x$--component other than the one it is contained in, and as $D\cap\st(x) = \emptyset$, it is necessarily the case that $D$ is equal to a union of $x$--components.
	Let $L$ be the set of vertices in $\st(x)$ that are adjacent to a vertex of $D$.
	Now, any vertex in $L$ must be in $\st(z)$ as $\st(z)$ separates $D$ and $x$ by assumption.
	Meanwhile, $L$ must also be contained in $\st(u)$ as otherwise it gives a two-edge path from $D$ to $z$ avoiding $\st(u)$.
	In particular,	this implies that $D$ must be a union of $u$--components, contradicting $\pi^z_D \in B$.
\end{proof}

The following lemma says we can delete some $Z_i$ to make all partial conjugations in $B$ trivial, while keeping some other $Z_{i'}$ so that the action of the surviving partial conjugations $\pi^u_E$ on $Z_i$ is tracked via its action on $Z_{i'}$.

\begin{lem}\label{lem:killing in case 2}
	We can reorder the sets $Z_i$ so that we can apply a projection map 
	$$p_3 \colon \im  p \to \Out(A_{\Lambda})$$
	where $\Lambda$ is the subgraph of $\Lambda_0$ induced by the vertices 
	$$Z_0\cup\cdots \cup Z_s \cup Y \cup \{x\}$$
	for some $s\leq r$, so that the following hold:
	\begin{enumerate}
		\renewcommand{\theenumi}{(\alph{enumi})}
		\renewcommand\labelenumi{\theenumi}
		\item \label{item:case 2 killing 1} For each $i$ such that $s<i\leq r$ there is some $j$ with $1\leq j\leq s$ 
		such that $Z_i$ and $Z_{j}$ lie in the same $u$--component  for each $u\in U$.
		\item \label{item:case 2 killing 2} All partial conjugations $\pi^z_D \in B$ are sent to the identity.
		\item \label{item:case 2 killing 3}The image of $\pi^x_{C_0}$ under $p_3$ is non-trivial.
	\end{enumerate}
\end{lem}

\begin{proof}
	The process to choose the $Z_i$ that are deleted is iterative, and we define $p_3$ as a composition of projection maps, each one obtained by killing some $Z_i$.
	
	Begin by setting $B=B_0$, and $\psi_0 \colon \im p \to \Out(A_{\Lambda_0})$ the identity map.
	We suppose $\Lambda_n$ has been obtained from $\Lambda_0$ by deleting some of the sets $Z_i$
	and the projection map $\psi_n \colon \im p \to \Out(A_{\Lambda_n})$ is well-defined.
	Let $B_n$ be the set of partial conjugations $\pi^z_D \in B$ that are not in the kernel of $\psi_n$.
	We will define $\psi_{n+1}$ so that the corresponding set $B_{n+1}$ is strictly smaller than $B_n$.
	
	Choose any $Z_i$ so that $z\in Z_i$ admits a partial conjugation $\pi^z_D$ in $B_n$.
	Since ${\psi}_n(\pi^z_D)$ is non-trivial, $Z_i$ is in $\Lambda_n$ and, up to an inner automorphism, we may assume  $x\notin D$.
	Furthermore $D$ must contain some set $Z_{i'}$ that is in $\Lambda_n$.
	By Lemma~\ref{lem:bad supports in y-cpnt}, $Z_i$ and $Z_{i'}$ are in the same $u$--component for each $u\in U$.
	Since we are in Case~\ref{case:not same y}, we must have $i' \ne 0$.
	
	Since $Z_i$ is not dominated by $x$, it cannot be dominated by any $Z_j$ by Lemma~\ref{lem:z_i leq z_j implies z_i leq x}, or by any $u\in U$.
	Hence we may delete $Z_i$ from $\Lambda_n$ to get a new graph $\Lambda_{n+1}$
	and define a projection map $\im \psi_n \to \Out(A_{\Lambda_{n+1}})$
	(compare with Lemma~\ref{lem:kill Z not dominated}).
	Compose this projection map with $\psi_n$ to get $\psi_{n+1} \colon \im p \to \Out(A_{\Lambda_{n+1}})$.
	In particular the partial conjugation $\pi^z_D$ is in the kernel of $\psi_{n+1}$ so is not included in $B_{n+1}$.
	
	Stop this process when we reach $n$ with $B_n$ empty.
	Then we set $p_3 = \psi_n$.
	Clearly item~\ref{item:case 2 killing 2} holds, by construction.
	
	For item~\ref{item:case 2 killing 1}, the construction yields, for each $i>s$ some $i'$ so that $Z_i$ and $Z_{i'}$ are in the same $u$--component for each $u\in U$.
	It may be that $i' > s$ (and so deleted to form $\Lambda$) but then we have some $i''$ so that $Z_i$ and $Z_{i''}$ are in the same $u$--component for each $u\in U$. 
	We can repeat this until, after finitely many steps, we find the required set $Z_j$ that is in $\Lambda$.
		
	Finally, for point~\ref{item:case 2 killing 3}, the image of $\pi^x_{C_0}$ would be trivial only if all components $Z_i$ have been killed except $Z_0$.
	This cannot happen because after each $Z_i$ is killed, there must be some $Z_{i'}$ left surviving, with $i' \ne 0$.
	This implies $s>0$.
\end{proof}

To conclude, in case~\ref{case:not same y} we define the homomorphism $q$ to be the composition
$$q = p_3 \circ p \colon \Sout^0(A_\Gamma) \to \Out(A_\Lambda).$$

\begin{ex}\label{ex-case3ctd}
	We continue Example~\ref{ex-case3}.
	This graph gives no partial conjugations in the set $B$, so the map $p_3$ is taken as the identity.
	We have
	 $$A_{\Lambda} \cong \Z^{c_0} \ast \Z^{c_1} \ast \Z^{c_2} \ast \Z^{c_3} \ast \Z^3$$	 
	 (where we use $c_i$ for the size of an equivalence class that can optionally be put in place of $z_i$).
\end{ex}

\begin{ex}\label{ex-case2ctd}
	We continue Example~\ref{ex-case2}.
	In this case the set $B$ is non-empty, and includes partial conjugations with multipliers $z_2,z_4$ and $z_5$. In particular, $B$ consists of
	\begin{itemize}
		\item $\pi^{z_2}_{C_1}$,	
		\item $\pi^{z_2}_{D}$ where $D$ contains $C_0,C_3,C_4,C_5$ and some vertices from $\st(x)$,
		\item $\pi^{z_4}_{C_5}$, 
		\item $\pi^{z_4}_E$ where $E$ contains $C_0,C_1,C_2,C_3$ and some vertices from $\st(x)$,
		\item $\pi^{z_5}_{C_4}$, and
		\item $\pi^{z_5}_{E}$.
	\end{itemize}
	Note that there are also partial conjugations by $z_2$ that acts on the vertices of $C_2 \setminus \st(z_2)$, and similarly for $z_4$ and $z_5$. However none of these are in the image of $p$.

	We can define $\psi_1$ by killing $z_2$. This will leave $B_1$ to contain the partial conjugations with multiplier $z_4$ or $z_5$.
	Then we define $\psi_2$ by also killing $z_4$.
	Immediately we see that the partial conjugations with multiplier $z_4$ are in the kernel,  
	but it also contains $\pi^{z_5}_{C_4}$.
	Finally, $\pi^{z_5}_{E}$ has becomes inner.
		
	We are left with $\Lambda$ consisting of vertices $x,y_1,y_2, z_0, z_1,z_3,z_5$,
	and $$A_{\Lambda} \cong \Z^{c_0} \ast \Z^{c_1} \ast \Z^{c_3} \ast \Z^{c_5} \ast \Z^3$$
	(where we use $c_i$ for the size of an equivalence class optionally put in place of $z_i$).
\end{ex}

\subsection{Applying the homological representation}\label{sec:virt ind image}

We compose the map $q$ from Section~\ref{sec:q} with the representation $\rho_\pi$ of Section~\ref{sec:VI homo rep} to get
$$
\varphi = \rho_\pi \circ q \colon \Out_\pi(A_\Gamma) \to \PGL(V_{-1}).
$$
More specifically, we use the restriction of $q$ to the finite index subgroup $\Out_\pi(A_\Gamma)$ that is the pre-image of $\Out_\pi(A_\Lambda)$.

The action of $\Aut_\pi(A_\Gamma)$ on $H_1(T;\Q)$ preserves the lattice $H_1(T;\Z)$,
so a finite index subgroup $\Aut_{\pi,\Z}(A_\Gamma)$ of $\Aut_\pi(A_\Gamma)$ preserves the lattice 
$\Z \x \oplus \Z \y_1 \oplus \cdots \oplus \Z\y_k \oplus \Z \z_1 \oplus \cdots \oplus \Z \z_s$
in $V_{-1}$.
Restricting $\varphi$ to the image of this subgroup in $\Out(A_\Gamma)$  we therefore get
$$
\varphi \colon \Out_{\pi,\Z}(A_\Gamma) \to \PGL(d+s,\Z)
$$
where $\Out_{\pi,\Z}(A_\Gamma)$ has finite index in $\Out(A_\Gamma)$.
We now verify that in each case the image of $\varphi$ is virtually indicable, implying the same of $\Aut(A_\Gamma)$.

\subsubsection{Case~\ref{case:same y}}

We have done nearly all the work needed to obtain a map onto $\Z$ from a finite index subgroup of $\Out(A_\Gamma)$ in this case.
Indeed, $V_{-1}$ has dimension 2, and the image of the partial conjugation $\pi^x_{C_0}$ is an infinite order element.
It follows that the image of $\Out_{\pi,\Z}(A_\Gamma)$ under $\varphi$ in this case is an infinite subgroup of a virtually free group.
Hence it is virtually free itself (possibly virtually $\Z$) and in particular we can obtain a map onto $\Z$ from a finite index subgroup of $\Out(A_\Gamma)$.
This proves Theorem~\ref{thm:no T special} in case~\ref{case:same y}.

\subsubsection{Case~\ref{case:not same y}}

Since the dimension of $V_{-1}$ is not necessarily 2 in this case, we need to use a more involved strategy to get a map onto $\Z$.

We claim that we can restrict everything in the image of $\varphi$ to a subspace of $V_{-1}$ so that the image is free abelian.
We begin with an example.

\begin{ex}
	We continue Example~\ref{ex-case3} (see also Example~\ref{ex-case3ctd}).
	We would like to restrict to some vectors $\z_i$ and $\x$, however 
	we have partial conjugations $\pi^{y_1}_{C_1\cup C_2}$ and $\pi^{y_2}_{C_2 \cup C_3}$
	which prevent us from (naively) doing this.
	These partial conjugations act as follows.
	$$	
	\begin{array}{lll}
	\varphi(\pi^{y_1}_{C_1\cup C_2}) ( \z_1 ) = \z_1 + 2\y_1 , & \varphi(\pi^{y_1}_{C_1\cup C_2}) ( \z_2 ) = \z_2 + 2\y_1, & \varphi(\pi^{y_1}_{C_1\cup C_2}) ( \z_3 ) = \z_3,  \\	
	\varphi(\pi^{y_2}_{C_2\cup C_3}) ( \z_1 ) = \z_1 & \varphi(\pi^{y_2}_{C_2\cup C_3}) ( \z_2 ) = \z_2 + 2\y_2 , & \varphi(\pi^{y_2}_{C_2\cup C_3}) ( \z_3 ) = \z_3 + 2\y_2.
	\end{array}
	$$
	Notice though that $\z = \z_1 - \z_2 + \z_3$ is fixed by both partial conjugations.
	In this example, we can restrict the action to the subspace $V$ spanned by $\z$ and $\x$.
	The partial conjugation $\pi^x_{C_0}$ acts as a transvection $\z \mapsto \z-2\x$, so the image of the representation in $\PGL(V) \cong \PGL(2,\Q)$ has infinite order in $\PGL(2,\Z)$.
	This is sufficient to imply virtual indicability of $\Out(A_\Gamma)$.	
\end{ex}

Consider the subspace $W$ spanned by the vectors $\z_i$ for $i=1,\ldots , s$.
As in Lemma~\ref{lem:span principal}, we associate to each partial conjugation $\pi^{u}_D$, with $u\in U$, a vector $\w_D$ in $W$. 
First, after multiplying by an inner automorphism we may assume that $C_0 \not\subset D$.
We define the vector $\w_D = \langle w_1 , \ldots , w_s \rangle$ in $W$ by
$$w_i = 
\begin{cases} 1 & \textrm{if $C_i \subseteq D$,} \\
0 & \textrm{otherwise.} 
\end{cases}$$
This gives us a set of vectors $\Pi$ defined as
$$\Pi = \{ \w_D \mid \textrm{$D$ is a $u$--component for some $u\in U$}\}.$$
As in Lemma~\ref{lem:span principal}, we can determine whether $(X,{C_0})$ is principal by inspecting the span of the set $\Pi$.
This requires more than a simple application of Lemma~\ref{lem:span principal} since we can only use $W$ to (immediately) see the effect on components $C_1,\ldots,C_s$. 
We need to ensure that the effect on the ``lost'' components $C_{s+1},\ldots, C_r$ can still be discerned.

\begin{lem}\label{lem:span 1}
	The vector $\langle 1 ,1 ,\ldots,1\rangle $ is in $\Span_\Q(\Pi)$  if and only if $(X,{C_0})$ is not principal.
\end{lem}

\begin{proof}
	The ``if'' direction follows from Lemma~\ref{lem:span principal}.
	
	As in the proof of Lemma~\ref{lem:span principal},
	consider a product of partial conjugations with multiplier $x$ and supports $D_i$ that are $u$--components for some $u\in U$.
	We write the product as
	$$
	(\pi^x_{D_1})^{\varepsilon_1} \cdots (\pi^x_{D_l})^{\varepsilon_l}
	$$
Up to an inner automorphism, and flipping the sign of $\varepsilon_i$, we may assume $C_0 \not\subset D_i$, and we then take the vector $\w_{D_i}\in\Pi$ defined above.	
	For $z\in C_1\cup\cdots\cup C_s$, we have 
	\begin{equation}\label{eq:product of pcs}
	(\pi^x_{D_1})^{\varepsilon_1} \cdots (\pi^x_{D_l})^{\varepsilon_l} (z)
	=
	x^{-\alpha_j} z x^{\alpha_j}
	\end{equation}
	where 
	$$\alpha_j 
	= \sum\limits_{\{ i \mid C_j \subset D_i\}} \varepsilon_i
	\ \
	\textrm{ and }
	\ \
	\langle \alpha_1,\ldots,\alpha_s \rangle = \sum\limits_{i=1}^l \varepsilon_i \w_{D_i}.$$
	The entry $\alpha_j$ records the effect on the component $C_j$,
	however we have not recorded the effect on each component $C_i$ for $i>s$.
	But by Lemma~\ref{lem:killing in case 2}, if we have deleted a class $Z_i$ from $\Gamma$ to create $\Lambda$, then there is some $Z_j$ such that 
	$Z_i$ and $Z_j$ are in the same component of $\Gamma\setminus \st(u)$ for each $u\in U$
	and $Z_j$ survives the trip down to $\Lambda$.
	We can therefore define $\alpha_i$ by using $\alpha_j$.
	Since we are in case~\ref{case:not same y}, we necessarily have $j \ne 0$.
	We then set $\alpha_i = \alpha_j$,
	and equation~\eqref{eq:product of pcs} extends to $z\in C_j$ for $1\leq j \leq r$.
	
	Hence, if $\langle 1, 1, \dots , 1\rangle  \in \Span_\Q (\Pi)$ then we have a product of such partial conjugations that is equal to $(\pi^x_{C_1\cup \cdots C_r})^n = (\pi^x_{C_0})^{-n}$ for some integer $n$, implying $(X,{C_0})$ is not principal.
\end{proof}

Before completing the proof of Theorem~\ref{thm:no T special} we look at one final example.

\begin{ex}
	We continue Example~\ref{ex-case2} (see also Example~\ref{ex-case2ctd}).
	We deleted vertices $z_2$ and $z_4$, so $W$ is $3$--dimensional, with basis $\z_1,\z_3,\z_5$.
	The set $\Pi$ will consist of vectors
	$\z_1 + \z_3 = \langle 1,1,0\rangle$ from $y_1$ and
	$\z_3 + \z_5 = \langle 0,1,1\rangle$ from $y_2$.
	The vector $\v=\z_1 - \z_3+\z_5 = \langle 1,-1,1\rangle $ is in $\Pi^\perp$, and indeed $\Pi^\perp = \Span_\Q(\v)$.
	
	Let $V = \Span_\Q(\v,\x)$.
	By direct calculations, or using Lemma~\ref{lem:inner product}, we see that $\pi^{y_1}_{C_1\cup C_2\cup C_3}$ acts trivially on $\v$ (and $\x$) so preserves $V$.
	The same is true of $\pi^{y_2}_{C_3\cup C_4\cup C_5}$.
	The partial conjugations by $z_2$ and $z_4$ were killed by the map $q$ since those vertices were killed, while the partial conjugation by $z_5$ was also killed because its support was $C_4$.
	The remaining partial conjugations are those with multiplier $x$, which preserve $V$:
	$$\varphi (\pi^x_{C_1}) (\v) = \v +2\x, \ \ 
	\varphi (\pi^x_{C_3}) (\v) = \v-2\x,$$
	$$\varphi (\pi^x_{C_5}) (\v) = \v+2\x, \ \ 
	\varphi (\pi^x_{C_0}) (\v) = \v+2\x.$$
	The other partial conjugations act trivially.
	The only transvections are by $x$ on $y_i$, both of which act trivially on $V$.
	We conclude that the image of $\varphi$ is $\Z$.
\end{ex}

We now complete the proof of Theorem~\ref{thm:no T special}.

Consider an automorphism in $\Aut_\pi(A_\Gamma)$. 
It can be written as $R \pi$, where 
$\pi$ is a product of partial conjugations and 
$R$ a product of transvections. 
This follows from the relations in Lemma~\ref{lem:conjugates of pc by transv}, writing the automorphism as a product of partial conjugations, transvections, and their inverses,
and then shuffling the partial conjugations (and their inverses) to the right.
Each partial conjugation is in $\Aut_\pi(A_\Gamma)$, so we can assume $R$ is too.

Since $(X,C_0)$ is not principal, Lemma~\ref{lem:span 1} tells us the span of $\Pi$ does not include $\langle 1,1,\ldots,1\rangle $.
	Its orthogonal complement is therefore non-trivial and we can
	define $V$ to be the subspace
	$$V=\Span_\Q(\Pi^\perp \cup \{ \x \} ).$$
	We claim that $V$ (more accurately, the associated projective space) is fixed by all partial conjugations except those with multiplier $x$, and is fixed by any product of transvections in $\Aut_\pi(A_\Gamma)$ (in particular, by $R$).
	
	For any $\v \in \Pi^\perp$, by Lemma~\ref{lem:inner product}, we get $\varphi(\pi^a_D)(\v) = \v$ for any $a\leq x$, $a\ne x$ and corresponding components $D$. 
	The vector $\x$ is sent ot $-\x$,
	so these partial conjugations fix the projective space associated to $V$.
	The remaining partial conjugations are those with multiplier in $z \in Z_i$ for some $i$.
	If $\pi^z_D$ is in the set $B$, then it is in the kernel of $\varphi$ by construction---either it is in the kernel of $p_3$ or it is mapped to an inner automorphism by $p_3$, which is then in the kernel of $\rho_\pi$.
	If it is  not in $B$, or not otherwise killed by $q$, then
	$1\leq i \leq s$ and $D$ is a union of $u$--components for some $u\in U$. 
	In particular $\w_D \in \Span_\Q(\Pi)$, so $\varphi(\pi^a_D)(\v) = \v$ for any $\v \in \Pi^\perp$.
	As previously, if $x\in D$, or not, the vector $\x$ is fixed.

	As for transvections, we claim that $R$ acts trivially on $V$.
	Consider $R_a^b$.
	By Lemma~\ref{lem:Zi dominates x implies case 1}, we have $a\in Y  \cup Z_0 \cup \cdots \cup Z_r$.
	If $a \in Y$ then $R_a^b$ acts trivially on $V$.
	So we assume $a \in Z_0 \cup \cdots \cup Z_r$.
	
	We will make use of the following observation.
			
\begin{lem}\label{lem:z<x z Pi}
	If $Z_i = \{z_i\}$ is dominated by $x$ then either 
	\begin{itemize}
		\item $i>0$ and $\z_i$ is in $\Pi$, or
		\item there is no vertex $u$ such that $z_i \leq u \leq x$.
	\end{itemize} 
	
	In particular, if either $i=0$ or $i>0$ and $\z_i \notin \Pi$ then we get a finite-index subgroup of $\Aut(A_\Gamma)$ mapping onto $\Z$.
\end{lem}

\begin{proof}
	The case when $i=0$ is dealt with by Lemma~\ref{lem:Z0 not domianted by x}, which implies
	condition~\ref{A} of \cite{Aramayona-MartinezPerez}, so the virtual surjection to $\Z$ follows from Theorem~\ref{thms:AMP}.	
	Now suppose $i>0$.
	If $z_i$ is dominated by any vertex $u\in U$ then $Z_i$ is a $u$--component, so $\z_i \in \Pi$.
	Hence we may assume that no vertex dominated by $x$ dominates $z_i$.
	This is again condition~\ref{A}.
\end{proof}

In light of Lemma~\ref{lem:z<x z Pi}, whenever $Z_i$ is dominated by $x$, we can assume $i>0$ and that $\z_i \in \Pi$.
We know from Lemma~\ref{lem:z_i leq z_j implies z_i leq x} that if $R_a^b$ is a transvection with $a\in Z_i$ and $b\in Z_j$ with $j\ne i$, then $a\leq x$.
Thus, unless $a$ and $b$ are both in $Z_i$, we must have that $a$ is dominated by $x$.
Then $\{a\} = Z_i$ for some $i$ and $\z_i \in \Pi$.
Thus if we take a vector $\v$ in $\Pi^\perp$ the $\z_i$--entry is zero for any $i$ for which there is a transvection $R_a^b$, with $a\in Z_i$, $b\notin Z_i$.
We therefore focus on the effect of $R$ on the entries of $\v$ that correspond to basis vectors $\z_i$ where $Z_i$ is not dominated by $x$.
The only way $R$ can act on $a \in Z_i$ is by sending it to an element of $\langle Z_i\rangle $, and furthermore, in order to be in $\Aut_\pi(A_\Gamma)$, $R(a)$ must be representable by a word of odd length on $Z_i \cup Z_i^{-1}$.
Thus $\varphi(R)$ sends $(1-g)e_{z_i}$ to $(1-g)(\kappa(1+g)e_{z_i} + e_{z_i}) = (1-g)e_{z_i}$ for some integer $\kappa$, and so $\z_i$ is fixed.
It follows that $R$ fixes each $\v \in \Pi^\perp$ and acts trivially on $V$.

To conclude, 
we can restrict to $V$ and get a new representation
$$\hat{\varphi} \colon \Out_{\pi,\Z}(A_\Gamma) \to \PGL(V).$$
The image of $\hat{\varphi}$ will be generated by the image of the partial conjugations $\pi^x_{C_i}$.
The image will therefore be a free abelian group of rank equal to $\dim(\Pi^\perp)$,
completing the proof of Theorem~\ref{thm:no T special}.

\section{Splitting the standard representation when there is no SIL}\label{sec:standard rep}

As defined in Section~\ref{sec:prelim:standard rep}, the standard representation of $\Out(A_\Gamma)$ is obtained by acting on the abelianisation of $A_\Gamma$ and we denote it by
$$\rho \colon \Out(A_\Gamma) \to \GL(n,\Z)$$
where $n$ is equal to the number of vertices in $\Gamma$.
We have a short exact sequence
\begin{equation}\label{eq:ses}
1 \to \IA_\Gamma \to \Sout^0(A_\Gamma) \to Q \to 1
\end{equation}
where $Q$ is a subgroup of $\SL(n,\Z)$.
The objective of this section is to show that under the assumption of no SIL, this short exact sequence splits, proving Proposition~\ref{propspecial}.

The structure of $Q$ is well-understood: it is (after conjugating $Q$ by a suitable permutation matrix) a block triangular matrix group.
This can be seen as follows.

\begin{conv}\label{convention vertex labelling}
	Enumerate the vertices of $\Gamma$ as $v_1,v_2,\ldots,v_n$ in such a way so that if $v_i \leq v_j$ then either $i\leq j$ or $v_i$ is equivalent to $v_j$, and so that equivalence classes of vertices are adjacent in this ordering.
\end{conv}

Under $\rho$, transvections map to elementary matrices.
We denote the image of $R_u^v$ by $E_u^v$; if $u=v_i$ and $v=v_j$, then $E_u^v = E_{ji}$, the matrix that differs from the identity by a $1$ in the $(j,i)$--entry.

Since the equivalence classes form clusters in this order,
we obtain a block structure in $Q$.
The blocks correspond to equivalence classes, and each diagonal block consists of matrices from $\SL(k,\Z)$, where $k$ is equal to the number of vertices in the corresponding equivalence class.
Matrices in $Q$ are lower block triangular by choice of the ordering on the vertices from Convention~\ref{convention vertex labelling} and the fact that $Q$ is generated by the set of elementary matrices $E_{ji}$ when $v_i\leq v_j$.

We now prove Proposition~\ref{propspecial}, determining that the short exact sequence \eqref{eq:ses} splits when there is no SIL in $\Gamma$.

\begin{proof}[Proof of Propsotion~\ref{propspecial}]
	Since $\Sout^0(A_\Gamma)$ is generated by partial conjugations and transvections and the kernel $\IA_\Gamma$ is generated by the partial conjugations, we know that $Q$ is generated by the image of the transvections, namely
	$$\{E_u^v \mid u\leq v\}.$$
	Using the matrix structure of $Q$ we get the following set of defining relators (see \cite[Proposition 4.11]{WadeThesis} for details): for $u\leq v\leq w$, and $x\leq y$ 
	\begin{enumerate}
		\renewcommand{\theenumi}{(\Alph{enumi})}
			\renewcommand\labelenumi{\theenumi}
		\item \label{splitrelator1} $[E_u^v , E_v^w] = E_u^w$ if $u\ne w$,
		\item \label{splitrelator2} $[E_u^v , E_x^y] = 1$ if $u\ne y$ and $v\ne x$,
		\item \label{splitrelator3} $(E_u^v(E_v^u)^{-1}E_u^v)^4 = 1$ if $u\ne v$ and $v\leq u$,
		\item \label{splitrelator4} $E_u^v (E_v^u)\m E_u^v E_v^u (E_u^v)\m E_v^u$ if $\{u,v\}$ is an equivalence class of size 2.
	\end{enumerate}
	To see the short exact sequence splits, define $\sigma \colon Q \to \Sout^0(A_\Gamma)$ by sending each $E_u^v$ to $R_u^v$.
	We need to check the four relators hold in the image of $\sigma$.
	
	Since $u\leq v\leq w$, Lemma~\ref{lem:domination pairs} implies $[v,w]=1$. Then direct calculation, left to the reader, verifies the relation $[R_u^v,R_v^w] = R_u^w$, so~\ref{splitrelator1} holds.
	
	For~\ref{splitrelator2}, if $u=x$ then $[v,y]=1$ by Lemma~\ref{lem:domination pairs}, and so $R_u^v$ and $R_u^y$ commute, as required.
	If $u\ne x$, then since also $u\neq y$ and $v \neq x$, the supports of $R_u^v$ and $R_x^y$ are disjoint and do not contain the multipliers. It follows that the transvections again commute.
	
	Finally, for both~\ref{splitrelator3} and~\ref{splitrelator4}, $u$ and $v$ are in the same equivalence class in $\Gamma$.
	Either this class has size at least 3, and so is abelian (since a non-abelian equivalence class of size at least 3 gives a SIL), or the equivalence class has size 2.
	We claim that in either case the subgroup $\langle R_u^v,R_v^u\rangle$ embeds into a copy of $\SL(n,\Z)$, where $n$ is the number of vertices in the equivalence class.
	
	Let $A$ denote the subgroup of $A_\Gamma$ generated by the equivalence class containing $u$ and $v$.
	Since $R_u^v$ and $R_v^u$ preserve the kernel of the projection map $\kappa\colon A_\Gamma \to A$ obtained by killing all vertices of $\Gamma$ not in this class
	we can define the factor map
	$$f \colon \langle R_u^v,R_v^u\rangle \to \Sout^0 ( A )$$
	so that $f(\Phi)(g) = \kappa(\Phi(g))$ for $\Phi \in \langle R_u^v,R_v^u\rangle$ and $g\in A$.
	It is clear that $\Phi(g)$ is in $A$, up to conjugacy, and thus $f(\Phi)(g) = \Phi(g)$.
	Thus $\Phi$ is in the kernel of $f$ only if $\Phi$ acts as an inner automorphism on $A$.
	If $A$ is abelian this is not possible for nontrivial $\Phi \in \langle R_u^v, R_v^u \rangle$, so $f$ is an embedding into $\Sout^0(A) \cong \SL(n,\Z)$ as claimed.
	On the other hand, if $A$ is not abelian, it must be non-abelian free of rank 2.
	In this case, using the fact that $\Gamma$ has no SIL, we must have that $A_\Gamma$ splits as a direct product $A \times B$.
	Indeed, $B$ is generated by the vertices of $\Gamma$ different to $u$ or $v$, and if any such vertex $x$ was not adjacent to $u$ and $v$ then we would obtain a SIL $(u,v\mids x)$.
	In particular, if $f(\Phi)$ is inner on $A$, then $\Phi$ must have been inner on $A_\Gamma$.
	Thus $f$ is an embedding into $\Sout^0(A) \cong \SL(2,\Z)$.
	
	With this claim, and since $f$ maps $R_u^v$ and $R_v^u$ to elementary matrices in $\SL(n,\Z)$, the relations \ref{splitrelator3} and~\ref{splitrelator4} (the latter when $n=2$) hold in $\SL(n,\Z)$ and hence also in $\langle R_u^v,R_v^u\rangle$ as required. 
\end{proof}

\begin{rem}\label{rem:split no sil} 
	We note that the short exact sequence may still split in cases when $\Gamma$ does have a SIL.
	For example, as long as $\Gamma$ has no SIL of the form  $(x,y\mids z)$ with $z\leq x,y$, then if all its equivalence classes are abelian, the sequence will still be split.
	This can be seen from the above proof, since the no SIL condition was used for three reasons. 
	One was in application of Lemma~\ref{lem:domination pairs}, which just requires the absence of the above type of SIL;
	a second was in deducing that equivalence classes of size at least three are abelian;
	and thirdly in the situation when we had a non-abelian equivalence class generating $A$.
\end{rem}

\section{Property (T) when there is no SIL}\label{sec:T}

In this section we show that, for a graph $\Gamma$ with no SIL, if Theorems~\ref{thms:AMP} and~\ref{thm:no T special} do not imply virtual indicability, then the outer automorphism group $\Out(A_\Gamma)$ has property (T).
This results in Theorem~\ref{thms:T}

For background  material concerning property~(T), we refer the reader to the book by Bekka, de la Harpe, and Valette \cite{PropT}.
Some key facts regarding property~(T) that we rely on are the following: 
\begin{itemize}
\item it passes to and from finite index subgroups \cite[Theorem~1.7.1]{PropT};
\item it is passed to quotients \cite[Theorem~1.3.4]{PropT};
\item and it is stable under short-exact sequences \cite[Proposition~1.7.6]{PropT}.
\end{itemize}
The following is central to our method.

\begin{lem}\label{lem:putting the pieces together}
	Suppose $H_1,\ldots,H_s$ are normal subgroups of $G$ and each has property~(T).
	Then $\langle H_1,\ldots , H_s\rangle$ has property~(T).
\end{lem}

\begin{proof}
	We use induction on $s$, with the case $s=1$ trivial.
	Since $\langle H_1, \ldots , H_s\rangle /H_1$ is a quotient of $\langle H_2 , \ldots, H_s\rangle $, it has property~(T) by induction and the fact that property~(T) passes to quotients.
	Stability of property (T) under short-exact sequences then implies $\langle H_1, \ldots , H_s\rangle$ has property~(T).
\end{proof}

The method to prove Theorem~\ref{thms:T} is then as follows.
We will decompose a finite index subgroup of $\IA_\Gamma$ into a direct product of subgroups $A_1,\ldots,A_s$, each generated by a subset of partial conjugations.
As $\Gamma$ has no SIL, each $A_i$ is abelian, and furthermore we will see that it is invariant under the action of $Q = \Sout^0(A_\Gamma) / \IA_\Gamma$.
In particular $Q\ltimes A_i$ is normal in $Q\ltimes A$, where $A = \langle A_1 , \ldots , A_s \rangle$ has finite index in $\IA_\Gamma$.
We will show that for each $i$ the group $Q\ltimes A_i$ has property~(T), allowing us to apply 
Lemma~\ref{lem:putting the pieces together}.
Then $\Out(A_\Gamma)$ inherits property~(T) from its finite index subgroup $Q\ltimes A$.

We now describe the decomposition of $\IA_\Gamma$ (up to a finite index subgroup).
Let $X$ be an equivalence class in $\Gamma$ and $C$ an $X$--component.
Recall the set $P^X_C$ is defined as
$$P^X_C = \{\pi^y_{C'} \mid y \geq X, C' = C \setminus \st(y)\}.$$
The set $P^X_C$ is a basis for a free abelian group,
which we denote $A^X_C$.

It is clear that the union of all subsets $P^X_C$ as $X$ and $C$ vary over all equivalence classes and corresponding components will contain all partial conjugations and therefore generate $\IA_\Gamma$ by Proposition~\ref{prop:Torelli}.
However, we restrict ourselves to consider only those sets $P^X_C$ for which $(X,C)$ is principal.
By Lemma~\ref{lem:virtual torelli} these will generate a finite-index subgroup of $\IA_\Gamma$.

\begin{lem}
	Let $X$ be an equivalence class in $\Gamma$ and $C$ an $X$--component.
	
	The subgroup $A_C^X =  \langle P^X_C \rangle$ is normal in $\Sout^0(A_\Gamma)$.
\end{lem}

\begin{proof}
	This follows from the relations in Lemma~\ref{lem:conjugates of pc by transv} under the assumption that there is no SIL.
\end{proof}

We aim to show that $Q \ltimes A^X_C$ has property (T) when $(X,C)$ is principal. 
This  is dome by showing that $Q\ltimes A^X_C$ is itself a block triangular matrix group, and in particular one of those covered by criteria set out in \cite[Section 4]{Aramayona-MartinezPerez} that determine when such groups have property (T).
We now introduce the relevant notation (which differs  slightly from that of \cite{Aramayona-MartinezPerez}).

Fix integers $m_1,m_2$ so that $m_1 < m_2$.
It may help when first reading this to assume $m_1 = 1$; 
in practice we will have either $m_1=1$ or $m_1=0$.
Let $V_1,\ldots,V_r$ be a partition of $I=[m_1,m_2]\cap\Z$ so that for each $x\in V_i$ and each $y\in V_j$, if $i< j$ then $x< y$.
Let $\Lambda$ be a directed graph with $r$ vertices labelled by $V_1,\ldots , V_r$.
Assume that there is an edge from the vertex labelled $V_i$ to the vertex labelled $V_j$ only if $i\leq j$, and that the edge relation is transitive: if there is an edge from $V_i$ to $V_j$, and another from $V_j$ to $V_k$, then there is an edge from $V_i$ to $V_k$.
Let $n_i$ be the size of $V_i$ and $n = m_2-m_1+1$.

We let $E_a^b$, for $a,b\in I$, denote the $n\times n$ elementary matrix $E_{ij}$ where $i=b+1-m_1$ and $j = a+1-m_1$. (Thus if $m_1=1$ then we have $E_a^b = E_{ba}$).
Define the group $\cal{H}_\Lambda$ to be the block triangular matrix generated by
	$$\{ E_a^b  \mid a\in V_i, b\in V_j \textrm{ and there is an edge from $V_i$ to $V_j$}\}.$$
The group $\cal{H}_\Lambda$ is a block lower-triangular matrix, with $i$--th diagonal block corresponding to $\SL(n_i,\Z)$,
and the $(i,j)$--th block, for $i\ne j$, being non-trivial if and only if there is an edge from $V_j$ to $V_i$.

We can identify $Q = \Out(A_\Gamma) / \IA_\Gamma$ with a group $\cal{H}_\Lambda$ as follows. 
Take $m_1=1$ and $m_2$ to be the number of vertices in $\Gamma$.
Order the vertices of $\Gamma$ as per Convention~\ref{convention vertex labelling}.
Let $V_1,\ldots , V_r$ be the equivalence classes so that if $i<j$, then if $v_a \in V_i$ and $v_b \in V_j$ then $a < b$.
To construct $\Lambda$, take the directed graph with $r$ vertices labelled by $V_i$ and add an edge from $V_i$ to $V_j$ whenever $V_i \leq V_j$.
Note that this includes edges from each $V_i$ to itself.

In the following we explain how to realise $Q\ltimes A^X_C$ as a matrix group of this form, starting with $Q \cong \cal{H}_\Lambda$.
You may think of the rows/columns of a matrix in $\cal{H}_\Lambda$ as corresponding to vertices of $\Gamma$.
To obtain the corresponding matrix representation of $Q\ltimes A^X_C$ we add a new row/column above/in front of the existing entries.
The new row/column can be thought of, roughly, as representing $C$.

\begin{lem}\label{lem:realise it as a matrix group}
	Let $X=V_i$ be an equivalence class in $\Gamma$ and let $C$ be an $X$--component.
	Construct $\hat\Lambda$ from $\Lambda$ by adding a vertex labelled by $V_{0} = \{0\}$, and adding edges from $V_0$ to itself, and to any $V_j$ where an edge from $V_i$ terminates.
	
	Then $Q\ltimes A^X_C \cong \cal{H}_{\hat\Lambda}$. 
\end{lem}

\begin{proof}
	The quotient map $\cal{H}_{\hat\Lambda} \to \cal{H}_\Lambda \cong Q$ that kills the first coordinate, corresponding to the integer $0$, has kernel $K$ isomorphic to $A^X_C$, seen as follows. 
	The kernel is generated by the matrices $E_0^b$ for any $b\in V_j$ where there is an edge from $V_i$ to $V_j$ in $\Lambda$, or equivalently so that $V_i \leq V_j$. 
	The isomorphism $K\cong A^X_C$ comes from identifying $E_0^b$ with $\pi^{v_b}_{C'}$ for each $\pi^{v_b}_{C'} \in P^X_C$. 
	Both groups $K$ and $A^X_C$ are free abelian of the same rank, namely $|P^X_C|$, and the above describes an identification of bases.
	
	By Lemma~\ref{lem:conjugates of pc by transv}, the action of $Q$ on $A^X_C$ agrees with the action of $\cal{H}_\Lambda$ on $K$, giving $Q\ltimes A^X_C \cong \cal{H}_{\hat\Lambda}$ as required.
\end{proof}

We are now ready to apply the result of \cite{Aramayona-MartinezPerez} that gives sufficient conditions for groups $\cal{H}_\Lambda$ to have property (T) in order to complete the proof of Theorem~\ref{thms:T}.
These conditions are as follows. Recall that $n_i$ is the size of $V_i$.

\begin{prop}[{\cite[Proposition~4.2]{Aramayona-MartinezPerez}}]\label{prop:block triangular and T}
	Let $\Lambda$ be constructed as above.
	Suppose the following conditions hold:
	\begin{enumerate}
		\renewcommand{\theenumi}{\normalfont ($\cal{H}\arabic{enumi}$)}
		\renewcommand\labelenumi{\theenumi}
		\item \label{am1} for each $i$ there is an edge from $V_i$ to itself,
		\item \label{am2} $n_i \ne 2$ for each $i$,
		\item \label{am3} whenever $n_i = n_j = 1$ and there is an edge from $V_i$ to $V_j$, with $i\ne j$, there is a third vertex $V_k$ and edges from $V_i$ to $V_k$ and from $V_k$ to $V_j$.
	\end{enumerate}
	Then $\cal{H}_\Lambda$ has property (T). 
\end{prop}

\begin{proof}[Proof of Theorem~\ref{thms:T}]
	Assume that $\Gamma$ has no SIL and that conditions~\ref{A} and \ref{B'} hold.
	Let $V_1 , \ldots , V_r$ be the equivalence classes of $\Gamma$.
	Note that property~\ref{A} implies each equivalence class of $\Gamma$ has size not equal to 2.

	Construct the graph $\Lambda$ from $\Gamma$ as above. 
	Condition~\ref{am1} holds in $\Lambda$ by construction, whil
	condition~\ref{A} implies that \ref{am2} and \ref{am3} also hold.
	
	Let $X=V_i$ and $C$ be an $X$--component, chosen so that $(X,C)$ is principal.
	Now construct $\hat\Lambda$ as described in Lemma~\ref{lem:realise it as a matrix group}.
	By construction, $\hat{\Lambda}$ inherits both \ref{am1} and \ref{am2} from $\Lambda$.
	For \ref{am3}, 
	if there is an edge from $V_0$ to $V_j$, for $j\ne i$, in $\hat\Lambda$, and $n_j= 1$, then there are also edges from $V_0$ to $V_i$ and from $V_i$ to $V_j$.
	This is sufficient since condition~\ref{B'} prevents us from having $n_i=1$.
	Proposition~\ref{prop:block triangular and T} therefore implies that $\cal{H}_{\hat\Lambda}$, and hence $Q \ltimes A^{X}_C$ by Lemma~\ref{lem:realise it as a matrix group}, has property~(T).
	
	To complete the proof, we apply Lemmas~\ref{lem:putting the pieces together} and~\ref{lem:virtual torelli}.
	Denote by $A$ the subgroup of $\IA_\Gamma$ generated by the sets $A^X_C$ when $(X,C)$ is principal.
	By Lemma~\ref{lem:putting the pieces together} we get that $Q\ltimes A$ has property~(T).
	Since $A$ has finite index in $\IA_\Gamma$ by Lemma~\ref{lem:virtual torelli}, so does $Q\ltimes A$ in $\Out(A_\Gamma)$, and the result follows.
\end{proof}

\bibliographystyle{alpha}
\bibliography{bibliography_out_raag_noSIL-T}

\medskip
\noindent {Andrew Sale}, 
University of Hawaii at Manoa, \\
{andrew@math.hawaii.edu}, \ \href{https://math.hawaii.edu/~andrew/}{https://math.hawaii.edu/$\sim$andrew/}
\medskip

\end{document}